\documentclass[english,twoside]{article}

\usepackage[T1]{fontenc}
\usepackage{babel,enumerate}
\usepackage{amsmath}
\usepackage{amsfonts}
\usepackage{amssymb}
\usepackage{amsthm}

\newtheorem{Theo}{Theorem}

\newtheorem{Prop}{Proposition}
\newtheorem{Cor}{Corollary}

\theoremstyle{definition}
\newtheorem{Def}{Definition}
\newtheorem{Not}{Notation}
\newtheorem{Rem}{Remark}

\newcommand\RR{\mathbb{R}}

\newcommand\CC{\mathbb{C}}

\newcommand\NN{\mathbb{N}}

\title{  Higher order infinitesimal freeness\footnote{This work was partially supported by the {\emph Agence Nationale de la
Recherche} grant ANR-08-BLAN-0311-03.}}

\author{M. F\'evrier\thanks{Institut de Math\'ematiques de Toulouse, Equipe de Statistique et Probabilit\'es,  F-31062 Toulouse Cedex 09. E-mail: maxime.fevrier@math.univ-toulouse.fr }} 

\date{}
\begin{document}
\maketitle

\begin{abstract}
We define higher order infinitesimal noncommutative 
probability space and infinitesimal non-crossing cumulant functionals. 
In this framework, we generalize to higher order the notion of 
infinitesimal freeness, via a vanishing of mixed cumulants condition. 
We also introduce and study some non-crossing partitions 
related to this notion. Finally, as an application, we show 
how to compute the successive derivatives of the free convolution 
of two time-indexed families of distributions from their 
individual derivatives. 
\end{abstract}

\vspace{6pt}

\begin{center}
{\bf\large 1. Introduction}
\end{center}
\setcounter{section}{1}

Free probability theory was introduced by Voiculescu in the eighties 
with motivations from operator algebras \cite{vdn}, but many connections 
to other fields of mathematics like random matrices (see \cite{voi}) or combinatorics appeared. 
The combinatorial side of free probability, as noticed by Speicher \cite{spe}, 
is linked to the convolution on the lattices of non-crossing partitions, 
which have been first studied by Kreweras \cite{kre}. Biane proved in \cite{bia} 
that these lattices of non-crossing partitions can be embedded 
into the Cayley graphs of the symmetric groups, 
also known as the type A in the classification of finite reflection groups. 
Reiner introduced non-crossing partitions related to other types in this classification \cite{rei}. 
\\
In \cite{bgn}, the authors showed that it is possible to build a free probability theory 
of type B, by replacing the occurences of the symmetric groups and the non-crossing partitions 
of type A by their type B analogues, namely the hyperoctaedral groups and the non-crossing partitions 
of type B. In their work, a central role is played by the boxed convolution 
which is a combinatorial operation having a natural type B analogue and describing 
the multiplication of two freely independent noncommutative random variables. 
The specificity of the boxed convolution of type B led the authors to define 
a noncommutative probability space of type B as a system 
$(\mathcal{A},\varphi ,\mathcal{V},f,\Phi )$, where
$(\mathcal{A},\varphi )$ is a noncommutative probability space, 
$\mathcal{V}$ is a complex vector space, 
$f : \mathcal{V} \longrightarrow \CC$ is a linear functional,
$\Phi : \mathcal{A} \times \mathcal{V} \times \mathcal{A} \longrightarrow \mathcal{V}$ 
is a two-sided action of $\mathcal{A}$ on $\mathcal{V}$. 
A type B noncommutative random variable is therefore 
a couple $(a,\xi )\in \mathcal{A} \times \mathcal{V}$, 
its distribution is $\CC^2$-valued and the non-crossing cumulant functionals 
of type B introduced in \cite{bgn} are also with values in $\CC^2$. 
An important remark is that the first component of 
a non-crossing cumulant of type B 
in $(\mathcal{A},\varphi ,\mathcal{V},f,\Phi )$ is simply  
a non-crossing cumulant of type A in $(\mathcal{A},\varphi )$. 
It follows that the notion of freeness of type B 
for $(\mathcal{A}_1, \mathcal{V}_1) , \ldots , (\mathcal{A}_m, \mathcal{V}_m)$ 
in $(\mathcal{A},\varphi ,\mathcal{V},f,\Phi )$, 
defined in \cite{bgn} in terms of moments 
to ensure that the vanishing of mixed cumulants of type B holds, 
implies the freeness of $\mathcal{A}_1, \ldots , \mathcal{A}_m$ 
in $(\mathcal{A},\varphi )$. The free additive convolution of type B, denoted by $\boxplus ^{(B)}$, 
which describes the distribution of the sum of two 
type B noncommutative random variables that are free of type B, 
is an operation on the set of couples $(\mu ,\mu ')$ of linear functionals 
on $\CC [X]$ satisfying $\mu (1)=1$ and $\mu '(1)=0$. 
Later, Popa stated in \cite{pop} type B versions of usual limit theorems 
and defined a S-transform for noncommutative random variables of type B.
\\
Recently, the analytic aspects of free probability theory of type B were investigated in \cite{bs} ; 
in particular, the authors outlined an interesting application of the free probability of type B 
that they called infinitesimal freeness : defining (when they exist) 
the zeroth and first derivatives at $0$ 
of a time-indexed family of distributions $(\mu _t)_{t>0}$
as the couple of distributions $(\mu ^{(0)}, \mu ^{(1)})$ defined by 
$$\mu ^{(0)}=\lim_{t\rightarrow 0} \mu _t$$
and
$$\mu ^{(1)}=\frac{d}{dt}_{|t=0}\mu _t=\lim_{t\rightarrow 0} \frac{1}{t}(\mu _t-\mu ^{(0)}),$$
they prove that, given two such time-indexed families of distributions 
$(\mu _t)_{t>0}$ and $(\nu _t)_{t>0}$, the zeroth and first derivatives at $0$ 
of $\mu _t \boxplus \nu _t$, denoted by $((\mu \boxplus \nu )^{(0)} , (\mu \boxplus \nu )^{(1)})$, 
satisfy : 
$$((\mu \boxplus \nu )^{(0)} , (\mu \boxplus \nu )^{(1)})=
(\mu ^{(0)}, \mu ^{(1)})\boxplus ^{(B)}(\nu ^{(0)}, \nu ^{(1)}).$$
\\
Following this insight, a new approach of free probability of type B 
was developed in \cite{fn}, named infinitesimal freeness. 
The equivalent structures considered there, simplifying and generalizing 
the noncommutative probability space of type B from \cite{bgn}, are the 
infinitesimal noncommutative probability space $(\mathcal {A},\varphi ,\varphi ')$ 
consisting in a noncommutative probability space $(\mathcal {A},\varphi )$ 
to which has been added another linear functional $\varphi '$ on $\mathcal {A}$ 
satisfying $\varphi '(1_\mathcal {A})=0$, and 
the scarce $\mathbb{G}$-noncommutative probability space $(\mathcal {A},\tilde \varphi )$, 
where $\tilde \varphi $ is a linear map 
which consolidates the two functionals $\varphi ,\varphi '$ 
in a single one from $\mathcal {A}$ 
into a two-dimensional Grassman algebra $\mathbb{G}$ generated 
by an element $\varepsilon $ which satisfies $\varepsilon ^2 =0$. 
A scarce $\mathbb{G}$-noncommutative probability space appears 
in the framework of a noncommutative probability space of type B 
$(\mathcal{A},\varphi ,\mathcal{V},f,\Phi )$ when one considers the link-algebra 
$\mathcal{A}\times \mathcal{V}$ together with the map $(\varphi , f)$. 
Infinitesimal freeness of unital subalgebras 
$\mathcal{A}_1, \ldots , \mathcal{A}_m$ of 
an infinitesimal noncommutative probability space $(\mathcal {A},\varphi ,\varphi ')$ 
is defined as the rewriting of the condition defining freeness of type B in 
a more general context. More precisely, $\mathcal{A}_1, \ldots , \mathcal{A}_m$ 
are infinitesimally free if whenever $i_1, \ldots , i_n \in \{ 1, \ldots , k \}$ are such that 
$i_1 \neq i_2, i_2 \neq i_3 , \ldots , i_{n-1} \neq i_n$, 
and $a_1 \in \mathcal{A}_{i_1}, \ldots , a_n \in \mathcal{A}_{i_n}$ 
are such that $\varphi (a_1) = \cdots = \varphi (a_n) = 0$, 
then $\varphi ( a_1 \cdots a_n ) = 0$ and
$$\varphi '( a_1 \cdots a_n ) = 
\left\{  \begin{array}{l}
\varphi (a_1 \, a_n) \varphi (a_2 \, a_{n-1}) \cdots 
       \varphi ( a_{(n-1)/2} \, a_{(n+3)/2} ) \cdot
       \varphi '( a_{(n+1)/2} ),                                \\
\mbox{$\ \ $} \ \ \mbox{ if $n$ is odd and $i_1 = i_n , 
      i_2 = i_{n-1}, \ldots , i_{(n-1)/2} = i_{(n+3)/2}$, }     \\
0, \ \mbox{ otherwise.}
\end{array}  \right.$$
It is clear from this definition that infinitesimally free unital 
subalgebras of an infinitesimal noncommutative probability space 
$(\mathcal {A},\varphi ,\varphi ')$ are in particular free in 
$(\mathcal {A},\varphi )$. A converse is proved in \cite{fn} : given 
free unital subalgebras $\mathcal{A}_1, \ldots , \mathcal{A}_m$ 
of a noncommutative probability space $(\mathcal {A},\varphi )$, 
$\mathcal{A}_1, \ldots , \mathcal{A}_m$ are infinitesimally free in 
the infinitesimal noncommutative probability space 
$(\mathcal {A},\varphi ,\varphi ')$, for instance when we set 
$\varphi ':=\varphi \circ D$, 
where $D : \mathcal {A}\longrightarrow \mathcal {A}$ is a derivation 
such that $\forall 1\leq i\leq m, D(\mathcal{A}_i)\subseteq \mathcal{A}_i$.
Moreover, a method is presented to obtain analogues in the framework 
of interest of an infinitesimal noncommutative probability space 
$(\mathcal {A},\varphi ,\varphi ')$ for results already established in 
usual free probability. This method is roughly to work in $(\mathcal {A},\tilde \varphi )$, 
where the computations are easy in the sense that the combinatorics is 
exactly the same as in a usual noncommutative probability space, and 
to take advantage of the equivalence between the structures 
$(\mathcal {A},\tilde \varphi )$ and $(\mathcal {A},\varphi ,\varphi ')$. 
This method is applied in \cite{fn} to find the right notion of infinitesimal 
non-crossing cumulant functional, and to compute the formulas for 
alternating products of infinitesimally free noncommutative random variables. 
These formulas make the non-crossing partitions of type B appear, 
as a reminder of the type B origin of infinitesimal freeness. 
The present work is in the lineage of \cite{fn}. 
\\
With the motivation to obtain higher order derivatives at $0$ 
of $\mu _t \boxplus \nu _t$ from those of $\mu _t$ and $\nu _t$, 
we generalize indeed to higher order the notion of 
infinitesimal noncommutative probability space from \cite{fn}, 
by adding to the noncommutative probability space $(\mathcal {A},\varphi ^{(0)})$ 
a certain number $k$ of other linear functionals 
$(\varphi ^{(i)})_{1\leq i\leq k}$ on $\mathcal {A}$ 
satisfying $\varphi ^{(i)}(1_\mathcal {A})=0$. 
Following the same idea as \cite{fn}, some formulas, the infinitesimal analogue of the 
free moment-cumulant formula for instance, 
will be simplified in the equivalent scarce $\mathcal{C}_k$ structure 
$(\mathcal {A},\tilde \varphi )$, 
where the $k+1$ linear functionals $(\varphi ^{(i)})_{0\leq i\leq k}$ 
are consolidated in a unique linear map $\tilde \varphi $, but with values in a certain 
$(k+1)$-dimensional algebra $\mathcal{C}_k$. The main benefit coming from this trick 
is that the formulas in $(\mathcal {A},\tilde \varphi )$ 
are the same as in usual free probability, 
with the only difference that they take place in the $(k+1)$-dimensional 
algebra $\mathcal{C}_k$ instead of the field of complex numbers. 
In what follows, we will continuously switch from the infinitesimal framework 
$(\mathcal{A} , (\varphi ^{(i)})_{0\leq i\leq k})$ 
which is the one of interest to the scarce $\mathcal{C}_k$-structure 
$(\mathcal{A} , \tilde \varphi )$ which 
is handy because the computations are easier in it.   
\\
As noticed above, in infinitesimal freeness from \cite{fn}, some formulas involving 
$\varphi '$ also involve the lattices of non-crossing partitions 
of type B, due to the link of infinitesimal freeness 
with free probability of type B pointed out in \cite{fn}. 
In higher order infinitesimal freeness, new non-crossing partitions 
appear in the formulas involving $\varphi ^{(k)}$. 
These so-called non-crossing partitions of type $k$, 
generalizing both non-crossing partitions of type A (corresponding to the case $k=0$) 
and type B (corresponding to the case $k=1$), are introduced and studied in Section 6.
\\
Our approach is in a sense the opposite of the approach in \cite{bgn}. 
Indeed, in \cite{bgn}, the authors substitute the symmetric groups 
by the hyperoctahedral groups, and by the way non-crossing partitions 
of type A by their type B analogues and thus they obtain 
the noncommutative probability space of type B. 
In the present work, we directly substitute the 
noncommutative probability space by the $k$-th order 
infinitesimal noncommutative probability space, and we look for the 
non-crossing partitions of type $k$ appearing this way. 
\\
Following this introduction, the paper is divided in seven 
other sections. In Section 2, we introduce 
the two equivalent notions of infinitesimal noncommutative probability space of order $k$ 
and of scarce $\mathcal{C}_k$-noncommutative probability space 
and discuss their relations with other structures. In Section 3, 
we introduce infinitesimal non-crossing 
cumulant functionals of order $k$, and define infinitesimal freeness of 
order $k$ by a condition of vanishing mixed cumulants. Section 4 
is devoted to the addition and multiplication of infinitesimally free variables. 
The formula expressing the 
infinitesimal cumulants of the product of two infinitesimally free 
noncommutative random variables may be written 
as a sum on certain non-crossing partitions generalizing the 
non-crossing partitions of type B reviewed in Section 5. 
These special non-crossing partitions, called non-crossing partitions 
of type $k$, and the boxed convolution operation associated to them 
are introduced and studied in Sections 6 and 7. Finally, we give in Section 8 
an important application of higher order infitesimal freeness : 
a recipe for computing the higher order derivatives of the free convolutions 
of two distributions. 

\begin{center}
{\bf\large 2. Infinitesimal noncommutative probability space of order $k$}
\end{center}

Throughout the paper, the integer $k\in \NN$ is fixed. 
In this section, we introduce the two equivalent structures 
of infinitesimal noncommutative probability space and of 
scarce $\mathcal{C}_k$ noncommutative probability space 
and we discuss their relations to previously defined structures.

\begin{center}
{\bf 2.1 Infinitesimal noncommutative probability space of order $k$}
\end{center}

The object of this subsection is to introduce the structure which is the framework 
for our notion of infinitesimal freeness of order $k$, 
namely the {\em infinitesimal noncommutative probability space of order $k$}. 

\begin{Def} \label{incps}
We call infinitesimal noncommutative probability space of order $k$ 
a structure $(\mathcal{A}, (\varphi ^{(i)})_{0\leq i\leq k})$ where $\mathcal{A}$ is a 
unital algebra over $\CC$, $\varphi ^{(0)}:\mathcal{A}\longrightarrow \CC$ 
is a linear map with $\varphi ^{(0)}(1_\mathcal{A})=1$, and 
$\varphi ^{(i)}:\mathcal{A}\longrightarrow \CC$, $1\leq i\leq k$, 
are linear maps with $\varphi ^{(i)}(1_\mathcal{A})=0$.
\end{Def}

\begin{Rem} 
The notion of infinitesimal noncommutative probability space of order $1$ coincides 
with the notion of infinitesimal noncommutative probability space introduced 
in \cite{fn}. The structure defined above is therefore a generalization of 
this latter object, and the use of the adjective infinitesimal is justified. 
\end{Rem}

An element $a\in (\mathcal{A}, (\varphi ^{(i)})_{0\leq i\leq k})$ 
of an infinitesimal noncommutative probability space of order $k$ 
is called an {\em infinitesimal noncommutative random variable of order $k$}. 
The {\em infinitesimal distribution of order $k$} of a $n$-tuple 
$(a_1, \ldots ,a_n)\in \mathcal{A} ^n$ of infinitesimal 
noncommutative random variables of order $k$ is the $(k+1)$-tuple 
$(\mu ^{(i)})_{0\leq i\leq k}$ of linear functionals on 
$\CC \langle X_1,\ldots ,X_n \rangle$ 
defined by : $$\mu ^{(i)}(P(X_1, \ldots ,X_n)) : =\varphi ^{(i)}(P(a_1, \ldots ,a_n)).$$ 
The range of infinitesimal distributions is the set of {\em infinitesimal laws of order $k$}, 
introduced below. 

\begin{Def} 
An infinitesimal law (of order $k$) on $n$ variables is a $(k+1)$-tuple of linear functionals 
$(\mu ^{(i)})_{0\leq i\leq k}$, where $\mu ^{(i)} : \CC \langle X_1,\ldots ,X_n \rangle \rightarrow \CC$ 
is defined on the algebra of noncommutative polynomials and satisfies $\mu ^{(i)}(1)=\delta _i^0$.
\end{Def}

For some purposes, it is handy to consider, instead of $k+1$ linear functionals 
as in Definition \ref{incps}, an equivalent unique linear map 
with values in a $(k+1)$-dimensional algebra. 
The relevant algebra, denoted by $\mathcal{C}_k$, is described below. 

\begin{center}
{\bf 2.2 The algebra $\mathcal{C}_k$}
\end{center}

In \cite{fn}, the two linear maps $\varphi $ and $\varphi '$ 
of an infinitesimal noncommutative probability space 
$(\mathcal{A}, \varphi , \varphi ')$ are 
consolidated in a single linear map $\tilde \varphi $ on $\mathcal{A}$
with values in the two-dimensional Grassman algebra $\mathbb{G}$ generated 
by an element $\varepsilon $ which satisfies $\varepsilon ^2 =0$ : 
$$\mathbb{G} = \{ \alpha + \varepsilon \beta \mid \alpha , \beta \in \CC \}.$$
This algebra has a quite natural $(k+1)$-dimensional generalization 
introduced below.

\begin{Def} 
Let $\mathcal{C}_k$ denote the $(k+1)$-dimensional complex algebra $\CC^{k+1}$ 
with usual vector space structure and multiplication given by the following rule: 
if $\alpha = (\alpha ^{(0)},\ldots ,\alpha ^{(k)})\in \mathcal{C}_k$ 
and $\beta = (\beta ^{(0)},\ldots ,\beta ^{(k)})\in \mathcal{C}_k$, 
then $$\alpha \cdot \beta =(\gamma ^{(0)},\ldots ,\gamma ^{(k)})$$ is defined by 
\begin{equation} \label{prod}
\gamma ^{(i)} : =\sum_{j=0}^i C_i^j \alpha ^{(j)}\beta ^{(i-j)}.
\end{equation}
\end{Def}

The algebra $\mathcal{C}_k$ is a unital complex commutative algebra. 
Its unit is $1_{\mathcal{C}_k}=(1,0,\ldots ,0)$. 
An element is invertible in the algebra $\mathcal{C}_k$ if and only if its first coordinate is non-zero.\\
The analogy between formula \eqref{prod} defining the product in $\mathcal{C}_k$ and the well-known 
Leibniz rule giving the recipe for computing the derivatives of the product of two smooth functions 
makes it easy to establish the formula for the product $\beta =\alpha _1 \cdots  \alpha _n$ 
of $n$ elements $\alpha _1,\ldots ,\alpha _n\in \mathcal{C}_k$. Precisely, 
if $\alpha _j= (\alpha _j^{(0)},\ldots ,\alpha _j^{(k)})$ and 
$\beta = (\beta ^{(0)},\ldots ,\beta ^{(k)})$, one has : 
$$\beta ^{(i)}=\sum_{\lambda \in \Lambda _{n,i}} C_i^{\lambda _1,\ldots ,\lambda _n} 
\prod_{j=1}^{n} \alpha _j^{(\lambda _j)},$$
where 
$$C_i^{\lambda _1,\ldots ,\lambda _n}=\frac{i!}{\lambda _1!\cdots \lambda _n!}$$
and
\begin{equation} \label{lambdaset}
\Lambda _{n,i}:=\{\lambda =(\lambda _1,\ldots ,\lambda _n)\in \NN^{n}\mid \sum_{j=1}^{n}\lambda _j=i\}.
\end{equation}
\\
There is an alternative description of the algebra $\mathcal{C}_k$ : 
it may be identified with the algebra of $(k+1)$-by-$(k+1)$ 
upper triangular Toeplitz matrices (with usual matricial operations) as follows : 
$$(\alpha^{(0)},\ldots ,\alpha ^{(k)})\simeq \left(\begin{array}{ccccc}\alpha ^{(0)}&\alpha ^{(1)}&\ldots &\frac{\alpha ^{(k-1)}}{(k-1)!}&\frac{\alpha ^{(k)}}{k!}\\
0&\alpha ^{(0)}&\ldots &\ldots &\frac{\alpha ^{(k-1)}}{(k-1)!}\\
\ldots &\ldots &\ldots &\ldots &\ldots \\
\ldots &\ldots &\ldots &\alpha ^{(0)}&\alpha ^{(1)}\\
0&0&\ldots &\ldots &\alpha ^{(0)}\\ 
\end{array} \right).$$
Consider $$\varepsilon :=\left(\begin{array}{ccccc}0&1&\ldots &0&0\\
0&0&\ldots &\ldots &0\\
\ldots &\ldots &\ldots &\ldots &\ldots \\
\ldots &\ldots &\ldots &0&1\\
0&0&\ldots &\ldots &0\\ 
\end{array} \right).$$
It is easy to compute the values of $\varepsilon ^i$ for $0\leq i\leq k+1$ ; 
in particular $\varepsilon ^{k+1}=0_{\mathcal{C}_k}$ and any 
element $\alpha = (\alpha ^{(0)},\ldots ,\alpha ^{(k)})\in \mathcal{C}_k$ 
may be uniquely decomposed 
\begin{equation} 
\alpha = \sum_{i=0}^{k} \alpha ^{(i)} \frac{\epsilon ^i}{i!}.
\end{equation}
The family $(\frac{\varepsilon ^i}{i!} , 0\leq i\leq k)$ is thus a linear basis 
of $\mathcal{C}_k$, to which we will refer as the canonical basis of $\mathcal{C}_k$. 
In particular, $\mathcal{C}_k\simeq \CC[\varepsilon ]=\CC_k[\varepsilon ]\simeq \CC[X]/(X^{k+1})$.
\\
In the definition of a usual noncommutative probability space, if one asks for the state 
to be $\mathcal{C}_k$-valued, one obtains a slightly different structure, introduced 
in the next section. 

\begin{center}
{\bf 2.3 Scarce $\mathcal{C}_k$-noncommutative probability space}
\end{center}

\begin{Def} 
By {\em scarce $\mathcal{C}_k$-noncommutative probability space}, 
we mean a couple $( \mathcal{A} , \tilde \varphi )$, 
where $\mathcal{A}$ is a unital algebra over $\CC$ and 
$\tilde \varphi : \mathcal{A} \to \mathcal{C}_k$ is a linear map
satisfying $\tilde \varphi (1_{\mathcal{A}}) = 1_{\mathcal{C}_k}$.
\end{Def}

\begin{Rem} 
The notion of scarce noncommutative probability space 
was introduced in \cite{oan}, but only the particular case 
of scarce $\mathbb{G}$-noncommutative probability space 
was considered there. This same structure has been studied later 
in \cite{fn} in connection with infinitesimal noncommutative probability space 
and free probability of type B. 
\end{Rem}

\begin{Rem} \label{association}
To any infinitesimal noncommutative probability space of order $k$ 
$(\mathcal{A}, (\varphi ^{(i)})_{0\leq i\leq k})$, we may associate 
a natural scarce $\mathcal{C}_k$-noncommutative probability space 
$( \mathcal{A} , \tilde \varphi )$, by putting 
\begin{equation} \label{decomp}
\tilde \varphi :=\sum_{i=0}^k \varphi ^{(i)} \frac{\varepsilon ^i}{i!}
\end{equation}
Reciprocally, given a scarce $\mathcal{C}_k$-noncommutative 
probability space $( \mathcal{A} , \tilde \varphi )$, 
the linear decomposition of $\tilde \varphi $ in the canonical basis 
of $\mathcal{C}_k$ (see equation \eqref{decomp}) gives rise to 
$k+1$ linear functionals $(\varphi ^{(i)})_{0\leq i\leq k}$, and 
consequently to an infinitesimal noncommutative probability space of order $k$ : 
$(\mathcal{A}, (\varphi ^{(i)})_{0\leq i\leq k})$.
\\
The equivalence between the infinitesimal noncommutative probability space of order $k$ 
$(\mathcal{A}, (\varphi ^{(i)})_{0\leq i\leq k})$ and its associated 
scarce $\mathcal{C}_k$-noncommutative probability space 
$( \mathcal{A} , \tilde \varphi )$ is fundamental in what follows. 
Indeed, we will continuously switch from one structure to the other, 
according to the principle that our interest is in the infinitesimal 
structure whereas the computations are easier in the scarce $\mathcal{C}_k$ 
structure, in the sense that they mimetize those from usual free probability. 
\end{Rem}

An element $a$ of a scarce $\mathcal{C}_k$-noncommutative 
probability space $(\mathcal{A}, \tilde \varphi )$ 
is called a {\em $\mathcal{C}_k$-noncommutative random variable}. 
We associate to such an $a\in \mathcal{A}$ the sequence of its 
{\em $\mathcal{C}_k$-valued moments} $(\tilde \varphi (a^n))_{n\in \NN^\ast }$. 
We call {\em $\mathcal{C}_k$-valued distribution} of $a$ the whole sequence of its moments, 
or equivalently, the linear map from $\CC[X]$ into $\mathcal{C}_k$ 
which maps any polynomial $P$ to $\tilde \varphi (P(a))$. 
One may find easier to collect all the 
$\mathcal{C}_k$-valued moments in a formal power series, as follows : 

\begin{Def} 
Let $\mathcal{C}$ be a unital commutative algebra over $\CC$.
We denote by $\Theta _{\mathcal{C}}^{(A)}$ the set of power series of the form 
$$f(z)=\sum_{n=1}^\infty  \alpha _nz^n,$$
where the $\alpha _n$'s are elements of $\mathcal{C}$.
\end{Def}

\begin{Def} 
Let $( \mathcal{A} , \tilde \varphi )$ be a scarce $\mathcal{C}_k$-noncommutative probability space.
The {\em $\mathcal{C}_k$-valued moment series} of $a\in \mathcal{A} $ is 
the power series $\tilde M_a\in \Theta _{\mathcal{C}_k}^{(A)}$ defined as follows:
$$\tilde M_a(z):=\sum_{n=1}^{\infty } \tilde \varphi (a^n) z^n.$$
\end{Def}

The notion of $\mathcal{C}_k$-valued distribution 
is easily generalized to $n$-tuples of variables : 

\begin{Def} 
The {\em $\mathcal{C}_k$-valued distribution} of a $n$-tuple $(a_1,\ldots ,a_n)\in \mathcal{A} ^n$ 
of $\mathcal{C}_k$-noncommutative random variables in 
a scarce $\mathcal{C}_k$-noncommutative probability space $(\mathcal{A}, \tilde \varphi )$ 
is the linear map $\tilde \mu _{(a_1,\ldots ,a_n)} : \CC \langle X_1,\ldots ,X_n \rangle \rightarrow \mathcal{C}_k$ 
defined by $$\tilde \mu _{(a_1,\ldots ,a_n)}(P(X_1,\ldots ,X_n)):=\tilde \varphi (P(a_1,\ldots ,a_n)).$$
\end{Def}

As mentioned in \cite{fn}, scarce $\mathbb{G}$-noncommutative probability space 
and infinitesimal noncommutative probability space provide a nice framework to 
do free probability of type B. The equivalent structures defined above 
are therefore the natural setting for generalizing free probability of type B. 
There is another structure linked to free probability of type B 
that one may find interesting to generalize here : the noncommutative probability space 
of type B, proposed in the original work on free probability of type B \cite{bgn}. 
Its natural generalization is the {\em noncommutative probability space of type $k$} : 

\begin{Def} 
By a noncommutative probability space of type $k$ we understand a system 
$(\mathcal{V}^{(0)},f^{(0)},\ldots ,\mathcal{V}^{(k)},f^{(k)},(\Phi _{i,j})_{0\leq i,j\leq k})$, where
$(\mathcal{V}^{(0)},f^{(0)})$ is a noncommutative probability space of type A,
$\mathcal{V}^{(i)}$, $1\leq i\leq k$, are complex vector spaces,
$f^{(i)}:\mathcal{V}^{(i)}\longrightarrow \CC$, $1\leq i\leq k$, are linear maps,
$\Phi _{i,j}:\mathcal{V}^{(i)}\times \mathcal{V}^{(j)}\longrightarrow \mathcal{V}^{(i+j)}$,$0\leq i,j\leq k$, 
are bilinear maps satisfying $$\Phi _{h+i,j}(\Phi _{h,i}(x,y),z)=\Phi _{h,i+j}(x,\Phi _{i,j}(y,z)),$$
$\forall h,i,j\in \NN, \forall x\in \mathcal{V}^{(h)}, \forall y\in \mathcal{V}^{(i)}, \forall z\in \mathcal{V}^{(j)}$.
\end{Def}

To make the preceding definition work, we put $\mathcal{V}^{(i)}=\{0\}$, when $i\geq k+1$. 
The following fact noticed in \cite{fn} still holds : 
noncommutative probability spaces of type $k$ are 
particular cases of scarce $\mathcal{C}_k$-noncommutative probability spaces. 
Indeed, given a noncommutative probability space of type $k$ 
$(\mathcal{V}^{(0)},f^{(0)},\ldots ,\mathcal{V}^{(k)},f^{(k)},(\Phi _{i,j})_{0\leq i,j\leq k})$, 
the direct product $\prod_{i=0}^{k} \mathcal{V}^{(i)}$ 
can be endowed with a complex unital algebra structure, 
via the maps $(\Phi _{i,j})_{0\leq i,j\leq k}$. 
This algebra, together with the linear map 
$\tilde \varphi (x_0,\ldots ,x_{k}):=(f^{(0)}(x_0),\ldots ,f^{(k)}(x_{k}))$, forms a scarce
$\mathcal{C}_k$-noncommutative probability space.
\\

There are natural equivalent notions of freeness on the structures introduced above, 
generalizing both infinitesimal freeness from \cite{bs} and \cite{fn} and 
freeness of type B from \cite{bgn}. In \cite{fn}, 
infinitesimal freeness in $(\mathcal{A},\varphi ,\varphi ')$
is defined by two conditions on the linear functionals $\varphi ,\varphi '$ ; 
its generalization to an infinitesimal noncommutative probability space of order $k$ denoted by 
$(\mathcal{A}, (\varphi ^{(i)})_{0\leq i\leq k})$ would require $k+1$ conditions on 
the linear functionals $(\varphi ^{(i)})_{0\leq i\leq k}$. 
Infinitesimal freeness from \cite{fn} being also equivalent to the vanishing of 
the infinitesimal non-crossing cumulants, we adopt this approach and 
define the infinitesimal freeness of order $k$ by the vanishing of some multilinear 
functionals, called infinitesimal non-crossing cumulant functionals of order $k$ 
and introduced in the next section. 

\begin{center} 
{\bf\large 3. Infinitesimal non-crossing cumulants of order $k$}
\end{center}

We begin this section by reviewing some background on non-crossing partitions.

\begin{center}
{\bf 3.1 Miscellaneous facts on non-crossing partitions of type A}
\end{center}

A partition $p$ of a finite set $X$ is a family of 
disjoint non-empty subsets of $X$, called the blocks of $p$, 
whose reunion is $X$. 
The set of blocks of a partition $p$ of $X$ will be denoted 
throughout these notes by $\mbox{bl}(p)$ ; its cardinal by $| p |$.\\
For $a$ and $b$ in $X$, we write $a\sim _p b$ and say that $a$ and $b$ are linked 
(in $p$) to denote that $a$ and $b$ are in the same block of the partition $p$ of $X$.
The set of partitions of a finite set $X$ together with the reverse refinement order 
($p\preceq q$ if every block of $p$ is contained in a block of $q$) is a lattice.\\
Now suppose $(X,\leq )$ is a totally ordered set.\\
A partition $p$ of $X$ is called non-crossing if, whenever you have $a<b<c<d$ in $X$ 
such that $a\sim _p c$ and $b\sim _p d$, then $a\sim _p b$.\\
The set $(NC^{(A)}(X),\preceq )$ of non-crossing partitions of $X$ 
together with the reverse refinement order is itself a lattice.
Its maximal element $1_X$ has $X$ as its only block ; its minimal element $0_X$ has every singleton as a block.\\
When $X=[m]:=\{1<\ldots <m\}$, we write $NC^{(A)}(m)$ instead of $NC^{(A)}([m])$.\\
A nice way to represent a non-crossing partition $p\in NC^{(A)}(m)$ is to view $1,\ldots ,m$ 
as equidistant clockwisely ordered points on a circle, 
and to draw for each block of $p$ the convex polygone whose vertices are the elements of this block. 
It is a necessary and sufficient condition for a partition to be non-crossing that 
the polygones built this way do not intersect.\\
Biane found in \cite{bia} a bijection between the set of non-crossing partitions of $[m]$ 
and the set of points lying on a geodesic in the Cayley graph of the symmetric group $S_m$ 
with generators the set of all transpositions. 
This bijection $t$ associates to any non-crossing partition $p\in NC^{(A)}(m)$ the permutation 
$t(p)\in S_m$ whose restriction to each block $V$ of $p$ is the trace of the cycle 
$(1,\ldots ,m)\in S_m$ on $V$. For $a\in [ m ] $, $t(p)(a)$ is called the neighbour of $a$ in $p$. 
Geometrically, it is the first point linked to $a$ 
that one meets when one goes clockwisely around the circle, starting from $a$.\\ 
Let us recall that the Kreweras complementation map, denoted by $\mbox{Kr}$, 
is the anti-isomorphism of the lattice $NC^{(A)}(m)$ of non-crossing partitions of $[m]$ 
introduced by Kreweras in \cite{kre} and defined in the following way : consider a copy 
$$\overline {[m]}:=\{\overline{1}<\ldots <\overline{m}\}$$ of $[m]$
and order the elements of $[m]\cup \overline{[m]}$ as follows : 
$$\{1<\overline{1}<\ldots <m<\overline{m}\}.$$
Given $p$ a non-crossing partition of $[m]$, $\mbox{Kr}(p)$ is the biggest (for the reverse refinement order) 
partition of  $\overline{[m]}$ such that $p \cup \mbox{Kr}(p)$ is a non-crossing partition of 
$[m]\cup \overline{[m]}$. See \cite{ns97} for a nice geometric construction of the Kreweras complement.

\begin{Rem} 
On $[m]\cup \overline{[m]}$, we could have considered the alternative order 
$$\{\overline 1<1<\overline 2<\ldots <\overline m<m\}.$$
This would have led to another anti-isomorphism of $NC^{(A)}(m)$, 
also called Kreweras complementation map and denoted $\mbox{Kr}'$, which turns out to be the inverse of $\mbox{Kr}$.
\end{Rem}

There is an important equality (see \cite{kre}) verified by the number of blocks of 
the Kreweras complement of a non-crossing partition: 

\begin{equation} \label{nbblocks}
| p |+| \mbox{Kr}(p) |=m+1, \forall p\in NC^{(A)}(m).
\end{equation}

Notice that, for $p\in NC^{(A)}(m)$, $\mbox{Kr}^2(p)$ can be easily described 
in the geometric representation given above : 
$\mbox{Kr}^2(p)$ is the anti-clockwise rotation of $p$ with angle $\frac{2\pi}{m}$.

We conclude this subsection by the introduction of a total order 
on the blocks of a fixed non-crossing partition $p$ of $[m]$.  

\begin{Def} 
Let $p \in NC^{(A)}(m)$, and $V,W\in \mbox{bl}(p)$.\\
$1^o$ $V$ is said to be nested in $W$ if $\min W<\min V\leq \max V<\max W$.\\
$2^o$ $V$ is said to be on the left of $W$ if $\max V<\min W$.\\
$3^o$ $V\sqsubset W \Leftrightarrow V$ is nested in $W$ or $V$ is on the left of $W$. 
\end{Def}

The proof of the next proposition is trivial and left to the reader.

\begin{Prop} 
$\sqsubset $ is a total order on $\mbox{bl}(p)$.
\end{Prop}

If $p\in NC^{(A)}(m)$, we have seen that $p\cup \mbox{Kr}(p)$ is 
a non-crossing partition of $[m]\cup \overline{[m]}$ in $m+1$ blocks.
These blocks will be listed in two different ways.\\ 
The first way is to list them all together in the increasing order $\sqsubset $ : 
we will write $\mbox{Mix}(p,i)$ for the $i$-th block of $p\cup \mbox{Kr}(p)$ 
in the increasing order $\sqsubset $, for $1\leq i\leq m+1$.\\
For some purposes, it is nice to list separately the blocks of $p$ and of $\mbox{Kr}(p)$, 
and we will write $\mbox{Sep}(p,i)$ to denote the $i$-th block of $p$ in the increasing order $\sqsubset $ 
if $1\leq i\leq | p |$ and to denote the $(i-| p |)$-th block of $\mbox{Kr}(p)$ 
in the increasing order $\sqsubset $ if $| p |+1\leq i\leq m+1$.\\
It is interesting to look at the first blocks in the two resulting lists : 
$\mbox{Mix}(p,1)$ is a singleton in $[m]\cup \overline{[m]}$, $\mbox{Sep}(p,1)$ is an interval in $[m]$. 
In particular, we can deduce the well-known fact that a non-crossing partition 
always owns an interval-block.

\begin{center}
{\bf 3.2 $\mathcal{C}_k$-non-crossing cumulant functionals}
\end{center}

In this subsection, we define non-crossing cumulant functionals in the framework of a 
scarce $\mathcal{C}_k$-noncommutative probability space by the free moment-cumulant 
formula from usual free probability, with the only difference that the computations take place 
in the algebra $\mathcal{C}_k$ instead of the field of complex numbers $\CC$. 
The following notations are commonly used 
in the combinatorial theory of free probability. 

\begin{Not} 
Let $(a_1,\ldots ,a_n)\in \mathcal{A}^n$, and let $V=\{v_1<\ldots <v_m\}\subseteq [n]$, then we denote 
$$(a_1,\ldots ,a_n) \mid V:=(a_{v_1},\ldots ,a_{v_m})\in \mathcal{A}^m.$$
For a family of multilinear maps $(r_n: \mathcal{A}^n \rightarrow \mathcal{C}_k)_{n=1}^\infty $, 
we define for any $n\in \NN$ and any $\pi \in NC^{(A)}(n)$ the $n$-linear functional 
$r_\pi : \mathcal{A}^n \rightarrow \mathcal{C}_k$ by 
$$r_\pi (a_1,\ldots ,a_n):=\prod _{V\in \pi } r_{|V|}((a_1,\ldots ,a_n) \mid V).$$
\end{Not}

\begin{Def} 
Let $( \mathcal{A} , \tilde \varphi )$ be a scarce $\mathcal{C}_k$-noncommutative probability space.
The {\em $\mathcal{C}_k$-non-crossing cumulant functionals} are a family of multilinear maps 
$(\tilde \kappa _n: \mathcal{A}^n \rightarrow \mathcal{C}_k)_{n=1}^\infty $, 
uniquely determined by the following equation : 
for every $n\geq 1$ and every $a_1,\ldots ,a_n\in \mathcal{A}$, 
\begin{equation} \label{fmc}
\sum_{p\in NC^{(A)}(n)}\tilde \kappa _p(a_1,\ldots ,a_n)=\tilde \varphi (a_1\cdots  a_n).
\end{equation}
\end{Def}

In free probability of type A, the formula above is known as 
the free moment-cumulant formula \cite{hp}. The only difference is 
that computations here take place in the unital commutative complex algebra $\mathcal{C}_k$ 
instead of $\CC$. However, the proofs (see \cite{nsbook}) 
of the following classical results remain valid 
in this setting. That is why we record them without proof.\\
For every $n\geq 1$ and every $a_1,\ldots ,a_n\in \mathcal{A}$ we have that:
\begin{equation} \label{invfmc}
\tilde \kappa _n(a_1,\ldots ,a_n)=\sum_{p\in NC^{(A)}(n)} \mbox{M\"ob}(p,1_n) \tilde \varphi _p(a_1,\ldots ,a_n),
\end{equation}
where \mbox{M\"ob} is the M\"obius function of the lattice of non-crossing partitions. 
Obviously, the multilinear maps 
$(\tilde \varphi _n: \mathcal{A}^n \rightarrow \mathcal{C}_k)_{n=1}^\infty $ 
implicitely used in formula \eqref{invfmc} 
are defined by $\tilde \varphi _n(a_1,\ldots ,a_n)=\tilde \varphi (a_1\cdots  a_n)$.

\begin{Prop} \label{cwsae}
One has that $\tilde \kappa _n (a_1, \ldots , a_n) = 0$
whenever $n \geq 2$, $a_1, \ldots , a_n \in \mathcal{A}$,
and there exists $1 \leq i \leq n$ such that $a_i \in \CC 1_{\mathcal{A}}$. 
\end{Prop}

\begin{Prop} \label{cwpae}
Let $x_1, \ldots , x_s$ be in $\mathcal{A}$ and consider some products 
of the form
\[
a_1 = x_1 \cdots x_{s_1}, \ a_2 = x_{s_1 +1} \cdots x_{s_2}, 
\ \ldots , \ a_n = x_{s_{n-1} +1} \cdots x_{s_n},
\]
where $1 \leq s_1 < s_2 < \cdots < s_n = s$. Then 
$$\tilde \kappa _n (a_1, \ldots , a_n) = 
\sum_{ \begin{array}{c}
{\scriptstyle \pi \in NC(s) \ such} \\
{\scriptstyle that \ \pi \vee \theta = 1_s}
\end{array}  } \ \ \tilde \kappa _{\pi} (x_1, \ldots , x_s),$$
where $\theta \in NC(s)$ is the partition : 
$$\theta = \{ \{ 1, \ldots , s_1 \}, \{ s_1 + 1, \ldots , s_2 \} , \ldots ,
\{ s_{n-1} + 1, \ldots , s_n \} \}.$$
\end{Prop}

Given a $\mathcal{C}_k$-noncommutative random variable $a\in ( \mathcal{A} , \tilde \varphi )$, 
the quantities $\tilde \kappa _n(a,\ldots ,a)$ are called its 
{\em $\mathcal{C}_k$-valued cumulants}, and they are collected in a power series : 

\begin{Def} 
Let $( \mathcal{A} , \tilde \varphi )$ be a scarce 
$\mathcal{C}_k$-noncommutative probability space.
The {\em $\mathcal{C}_k$-valued R-transform} of $a\in \mathcal{A} $ 
is the power series $\tilde R_a\in \Theta _{\mathcal{C}_k}^{(A)}$ defined as follows :
$$\tilde R_a(z):=\sum_{n=1}^{\infty } \tilde \kappa _n(a,\ldots ,a) z^n.$$
\end{Def}

Following the well-known result of \cite{spe} stating roughly that, 
in a usual noncommutative probability space, subsets are free 
if and only if they satisfy the vanishing of mixed cumulants 
condition, we generalize this condition to our setting : 

\begin{Def} 
Let $( \mathcal{A} , \tilde \varphi )$ be a scarce 
$\mathcal{C}_k$-noncommutative probability space 
and $\mathcal{M}_1,\ldots ,\mathcal{M}_n$ be 
subsets of $\mathcal{A} $.
We say that $\mathcal{M}_1,\ldots ,\mathcal{M}_n$ 
have {\em vanishing mixed $\mathcal{C}_k$-cumulants} if
$$\tilde \kappa _m(a_1,\ldots ,a_m)=0$$
whenever $a_1\in \mathcal{M}_{i_1},\ldots ,a_m\in \mathcal{M}_{i_m}$ 
and $\exists 1\leq s<t\leq m$, such that, $i_s\not=i_t$.
\end{Def}

As announced, infinitesimal freeness of order $k$ is defined 
by the vanishing of mixed $\mathcal{C}_k$-cumulants condition. More precisely : 

\begin{Def} \label{f}
We will say that subsets $\mathcal{M}_1,\ldots ,\mathcal{M}_n\subseteq \mathcal{A} $ 
of a scarce $\mathcal{C}_k$-noncommutative probability space $( \mathcal{A} , \tilde \varphi )$ 
are {\em infinitesimally free of order $k$} if they have vanishing mixed $\mathcal{C}_k$-cumulants. 
\end{Def}

\begin{Rem} 
Using a classical argument in free probability, one can prove that, 
if $\mathcal{A}_1,\ldots ,\mathcal{A}_n$ are unital subalgebras 
which are infinitesimally free of order $k$ 
in a scarce $\mathcal{C}_k$-noncommutative probability space 
$( \mathcal{A} , \tilde \varphi )$, then one has : 
$$\tilde \varphi (a_1 \cdots  a_m)=0$$
whenever $a_1\in \mathcal{A}_{i_1},\ldots ,a_m\in \mathcal{A}_{i_m}$ 
with $i_1\not=\ldots \not=i_m$ satisfy 
$\tilde \varphi (a_1)=\ldots =\tilde \varphi (a_m)=0$.\\
The converse in our $\mathcal{C}_k$-valued situation is not true, 
because one cannot use 
the nice "centering trick", as noticed in $\cite{fn}$ Remark 4.9. 
\end{Rem}

In the next subsection, we switch to the infinitesimal framework, 
and define infinitesimal non-crossing cumulant functionals, with 
the intuition that they should appear as the coefficients in the decomposition 
of the $\mathcal{C}_k$-non-crossing cumulant functionals in the 
canonical basis of $\mathcal{C}_k$. 

\begin{center}
{\bf 3.3 Infinitesimal non-crossing cumulant functionals}
\end{center}

In this short subsection, we focus on an infinitesimal 
noncommutative probability space of order $k$ structure 
$(\mathcal{A}, (\varphi ^{(i)})_{0\leq i\leq k})$. 
The aim is to define cumulants and freeness in this setting, 
in a consistent way with the last subsection. 
For convenience, we will use the following notation : 

\begin{Not} 
For a family of multilinear maps $(r_n^{(i)}: \mathcal{A}^n \rightarrow \CC, 0\leq i\leq k)_{n=1}^\infty $, 
we define for any $n\in \NN$, any $\pi = \{V_1\sqsubset \cdots \sqsubset V_h\} \in NC^{(A)}(n)$ 
and any $\lambda \in \Lambda _{n,h}$ (defined by \eqref{lambdaset}) the $n$-linear functional 
$r_\pi ^{(\lambda )} : \mathcal{A}^n \rightarrow \CC$ by 
$$r_\pi ^{(\lambda )}(a_1,\ldots ,a_n):=\prod _{i=1}^h r_{|V_i|}^{(\lambda _i)}((a_1,\ldots ,a_n) \mid V_i).$$
\end{Not}

The underlying idea is to consider the $\mathcal{C}_k$-non-crossing cumulant functionals 
$(\tilde \kappa _n: \mathcal{A}^n \rightarrow \mathcal{C}_k)_{n=1}^\infty $ 
in the associated scarce $\mathcal{C}_k$-noncommutative probability space 
$( \mathcal{A} , \tilde \varphi )$ (see Remark \ref{association}), 
and then to define the required $n$-th infinitesimal non-crossing 
cumulant functionals as the $n$-linear forms appearing as coefficients in the linear decomposition of 
$\tilde \kappa _n: \mathcal{A}^n \rightarrow \mathcal{C}_k$ 
in the canonical basis of $\mathcal{C}_k$. This leads to the 
following definition: 

\begin{Def} 
Let $(\mathcal{A}, (\varphi ^{(i)})_{0\leq i\leq k})$ 
be an infinitesimal noncommutative probability space of order $k$.
The {\em infinitesimal non-crossing cumulant functionals of order $k$} are a family of multilinear maps 
$(\kappa _n^{(i)}: \mathcal{A}^n \rightarrow \CC , 0\leq i\leq k)_{n=1}^\infty $, 
uniquely determined by the following equation : 
for every $n\geq 1$, every $0\leq i\leq k$ and every $a_1,\ldots ,a_n\in \mathcal{A}$ we have that:
\begin{equation} \label{inffmc}
\sum_{\substack{p\in NC^{(A)}(n)\\p:=\{V_1,\ldots ,V_h \} }} \sum_{\lambda \in \Lambda _{h,i}} 
C_i^{\lambda _1,\ldots ,\lambda _h} \kappa _{p}^{(\lambda )}(a_1,\ldots ,a_n)=
\varphi ^{(i)}(a_1\cdots  a_n).
\end{equation}
\end{Def}

Infinitesimal freeness in the framework of 
an infinitesimal noncommutative probability space of order $k$ 
is obviously defined by the vanishing of mixed  
infinitesimal cumulants.

\begin{Def} \label{inff}
We will say that subsets $\mathcal{M}_1,\ldots ,\mathcal{M}_n$ of 
an infinitesimal noncommutative probability space of order $k$ 
are {\em infinitesimally free of order $k$} if they have 
{\em vanishing mixed infinitesimal cumulants}, which means that, for each $0\leq i\leq k$, 
$$\kappa _m^{(i)}(a_1,\ldots ,a_m)=0$$
whenever $a_1\in \mathcal{M}_{i_1},\ldots ,a_m\in \mathcal{M}_{i_m}$ 
and $\exists 1\leq s<t\leq m$, such that, $i_s\not=i_t$.
\end{Def}

\begin{Rem} 
It is straightforward to check, using formula \eqref{inffmc}, that the 
infinitesimal non-crossing cumulant functionals of an 
infinitesimal noncommutative probability space of order $k$ 
are indeed linked to  
the $\mathcal{C}_k$-non-crossing cumulant functionals of the associated 
scarce $\mathcal{C}_k$-noncommutative probability space by : 
\begin{equation} \label{decompcum}
\tilde \kappa _n=\sum_{i=0}^k \kappa _n^{(i)} \frac{\varepsilon ^i}{i!}.
\end{equation} 
A consequence of formulas \eqref{invfmc} and \eqref{decompcum} 
is the validity of the following inverse formula:
\begin{equation} \label{infinvfmc}
\kappa _n^{(i)}(a_1,\ldots ,a_n)=
\sum_{\substack{p\in NC^{(A)}(n)\\p:=\{V_1,\ldots ,V_h \} }}\sum_{\lambda \in \Lambda _{h,i}}\mbox{M\"ob}(p,1_n)
C_i^{\lambda _1,\ldots ,\lambda _h} \varphi _{p}^{(\lambda )}(a_1,\ldots ,a_n) , 
\end{equation} 
and of the following proposition : 

\begin{Prop} 
One has that $\kappa _n^{(i)} (a_1, \ldots , a_n) = 0$
whenever $0\leq i\leq k$, $n \geq 2$, $a_1, \ldots , a_n \in \mathcal{A}$,
and there exists $1 \leq j \leq n$ such that $a_j \in \CC 1_{\mathcal{A}}$. 
\end{Prop}

Another consequence of relation \eqref{decompcum} is that subsets 
$\mathcal{M}_1,\ldots ,\mathcal{M}_n$ of an infinitesimal 
noncommutative probability space of order $k$ 
are infinitesimally free of order $k$ if and only if they are
infinitesimally free of order $k$ in the associated scarce 
$\mathcal{C}_k$-noncommutative probability space.
\end{Rem}

\begin{Rem} \label{inffimpliesf}
Let $(\mathcal{A}, (\varphi ^{(i)})_{0\leq i\leq k})$ 
be an infinitesimal noncommutative probability space of order $k$, 
and consider its infinitesimal non-crossing cumulant functionals 
$(\kappa _n^{(i)}: \mathcal{A}^n \rightarrow \CC , 0\leq i\leq k)_{n=1}^\infty $.
It is interesting to notice that the multilinear maps 
$(\kappa _n^{(0)}: \mathcal{A}^n \rightarrow \CC )_{n=1}^\infty $ 
and $(\kappa _n^{(1)}: \mathcal{A}^n \rightarrow \CC )_{n=1}^\infty $
are respectively the usual non-crossing cumulant functionals in the 
noncommutative probability space $(\mathcal{A}, \varphi ^{(0)})$ 
and the infinitesimal non-crossing cumulant functionals of \cite{fn} in the 
infinitesimal noncommutative probability space $(\mathcal{A}, \varphi ^{(0)}, \varphi ^{(1)})$. 
This implies that subsets that are infinitesimally free of order $k$ 
are in particular free in $(\mathcal{A}, \varphi ^{(0)})$ and 
infinitesimally free in $(\mathcal{A}, \varphi ^{(0)}, \varphi ^{(1)})$ 
in the sense of \cite{fn}.\\
Infinitesimal freeness of unital subalgebras in \cite{fn}, 
as well as freeness of type B in \cite{bgn}, 
is defined in terms of moments. Section 8 will provide such 
a characterization of the infinitesimal freeness of order $k$ 
of unital subalgebras of an infinitesimal noncommutative probability space of order $k$ 
in terms of moments.
\end{Rem}

As stated in Remark \ref{inffimpliesf}, infinitesimal freeness of order $k$ of unital subalgebras 
$\mathcal{A}_1,\ldots ,\mathcal{A}_n\subseteq (\mathcal{A}, (\varphi ^{(i)})_{0\leq i\leq k})$ 
of an infinitesimal noncommutative probability space of order $k$ 
implies freeness of $\mathcal{A}_1,\ldots ,\mathcal{A}_n$ in the 
noncommutative probability space $(\mathcal{A}, \varphi ^{(0)})$. 
Conversely, is it possible to "upgrade" freeness of given unital subalgebras 
of a noncommutative probability space to infinitesimal freeness of order $k$ ? 
This question is discussed in the next subsection. 

\begin{center}
{\bf 3.4 Upgrading freeness to infinitesimal freeness of order $k$}
\end{center}

Given a noncommutative probability space $(\mathcal{A},\varphi )$ 
and free unital subalgebras $\mathcal{A}_1,\ldots ,\mathcal{A}_n$ of $\mathcal{A}$, 
the question of how to build a linear form $\varphi '$ on $\mathcal{A}$ 
such that $\mathcal{A}_1,\ldots ,\mathcal{A}_n$ are infinitesimally free in 
the infinitesimal noncommutative probability space $(\mathcal{A},\varphi ,\varphi ')$ 
is adressed in \cite{fn}. 
Among the answers given there, there is the idea to define $\varphi ':=\varphi \circ D$, 
where $D$ is a derivation of $\mathcal{A}$ 
(a linear map $D : \mathcal{A}\longrightarrow \mathcal{A}$ satisfying 
$\forall a,b\in \mathcal{A}, D(a\cdot b)=D(a)\cdot b+a\cdot D(b)$) 
such that $D(\mathcal{A}_j)\subseteq \mathcal{A}_j$ for each $1\leq j\leq n$. 
We examine the question of how to build linear forms $\varphi ^{(1)},\ldots ,\varphi ^{(k)}$ on $\mathcal{A}$ 
such that $\mathcal{A}_1,\ldots ,\mathcal{A}_n$ are infinitesimally free of order $k$ in 
the infinitesimal noncommutative probability space $(\mathcal{A},\varphi ,\varphi ^{(1)},\ldots ,\varphi ^{(k)})$. 
The natural idea consisting in defining $\varphi ^{(i)}:=\varphi \circ D^{i}$ 
where $D$ is a derivation of $\mathcal{A}$ such that $D(\mathcal{A}_j)\subseteq \mathcal{A}_j$ 
for each $1\leq j\leq n$ is a possible answer, as proved below : 

\begin{Prop} \label{upgrade}
Let $(\mathcal{A},\varphi )$ be a noncommutative probability space 
and let $D : \mathcal{A}\rightarrow \mathcal{A}$ be a derivation. 
Define $\varphi ^{(i)}:=\varphi \circ D^{i}$. 
Let the infinitesimal non-crossing cumulant functionals associated to 
$(\mathcal{A},\varphi ,\varphi ^{(1)},\ldots ,\varphi ^{(k)})$ be denoted 
by $(\kappa _n^{(i)}: \mathcal{A}^n \rightarrow \CC , 0\leq i\leq k)_{n=1}^\infty $. 
Then, for every $n\geq 1, 0\leq i\leq k$ and $a_1,\ldots ,a_n\in \mathcal{A}$ one has 
$$\kappa _n^{(i)}(a_1,\ldots ,a_n)=\sum _{\lambda \in \Lambda _{n,i}} C_i^{\lambda _1,\ldots ,\lambda _n} 
\kappa _n(D^{\lambda _1}(a_1),\ldots ,D^{\lambda _n}(a_n)).$$
\end{Prop}

\begin{proof}
Define the family of multilinear functionals 
$(\eta _n^{(i)}: \mathcal{A}^n \rightarrow \CC , 0\leq i\leq k)_{n=1}^\infty $ 
by the following formulas : for every $n\geq 1, 0\leq i\leq k$ and $b_1,\ldots ,b_n\in \mathcal{A}$
$$\eta _n^{(i)}(b_1,\ldots ,b_n)=\sum _{\lambda \in \Lambda _{n,i}} C_i^{\lambda _1,\ldots ,\lambda _n} 
\kappa _n(D^{\lambda _1}(b_1),\ldots ,D^{\lambda _n}(b_n)).$$
Our aim is then to prove that, for every $n\geq 1, 0\leq i\leq k$, $\eta _n^{(i)}=\kappa _n^{(i)}$. 
We verify that the functionals $(\eta _n^{(i)}, 0\leq i\leq k)_{n=1}^\infty $ satisfy the
equations \eqref{inffmc} defining the infinitesimal non-crossing cumulant functionals. 
The left-hand side of this formula writes : 
\begin{equation} \label{sum}
\sum_{\substack{p\in NC^{(A)}(n)\\p:=\{V_1,\ldots ,V_h \} }}\sum_{\lambda \in \Lambda _{h,i}} 
C_i^{\lambda _1,\ldots ,\lambda _h}\eta _{p}^{(\lambda )}(a_1,\ldots ,a_n).
\end{equation}
Each $\eta _{|V_j|}^{(\lambda _j)}((a_1,\ldots ,a_n) \mid V_j)$ in the latter is a sum indexed by 
$\Lambda _{|V_j|,\lambda _j}$, involving variables $a_i, i\in V_j$. 
Given $p:=\{V_1,\ldots ,V_h \}\in NC^{(A)}(n)$, there is a very natural bijection between 
$\{(\lambda ,(\lambda ^1,\ldots ,\lambda ^h))\in \Lambda _{h,i}\times \Lambda _{n,i}\mid \lambda ^j\in \Lambda _{|V_j|,\lambda _j}\}$ 
and the set $\Lambda _{n,i}$. 
Thus, the quantity \eqref{sum} rewrites : 
$$\sum_{p\in NC^{(A)}(n)}\sum_{\lambda \in \Lambda _{n,i}} 
C_i^{\lambda _1,\ldots ,\lambda _n}\kappa _{p}(D^{\lambda _1}(a_1),\ldots ,D^{\lambda _n}(a_n)).$$
By exchanging the summation signs, the usual free moment-cumulant formula appears, and one obtains : 
\begin{equation} \label{leib}
\sum_{\substack{p\in NC^{(A)}(n)\\p:=\{V_1,\ldots ,V_h \} }}\sum_{\lambda \in \Lambda _{h,i}} 
C_i^{\lambda _1,\ldots ,\lambda _h} \eta _{p}^{(\lambda )}(a_1,\ldots ,a_n)=
\sum_{\lambda \in \Lambda _{n,i}} C_i^{\lambda _1,\ldots ,\lambda _n} 
\varphi (D^{\lambda _1}(a_1)\cdots  D^{\lambda _n}(a_n)).
\end{equation}
Using Leibniz rule in the right-hand side of \eqref{leib}, one may conclude : 
\begin{eqnarray*}
\sum_{\substack{p\in NC^{(A)}(n)\\p:=\{V_1,\ldots ,V_h \} }}\sum_{\lambda \in \Lambda _{h,i}} 
C_i^{\lambda _1,\ldots ,\lambda _h} \eta _{p}^{(\lambda )}(a_1,\ldots ,a_n)\\
=\varphi (\sum_{\lambda \in \Lambda _{n,i}} C_i^{\lambda _1,\ldots ,\lambda _n} 
D^{\lambda _1}(a_1)\cdots  D^{\lambda _n}(a_n))\\
=\varphi (D^i(a_1\cdots  a_n))\\
=\varphi ^{(i)}(a_1\cdots  a_n).
\end{eqnarray*}
\end{proof}

\begin{Cor} 
In the notations of Proposition \ref{upgrade}, let $\mathcal{A}_1,\ldots ,\mathcal{A}_n$ 
be unital subalgebras of $\mathcal{A}$ which are free in $(\mathcal{A}, \varphi )$, 
and such that $D(\mathcal{A}_j)\subseteq \mathcal{A}_j$ for $1\leq j\leq n$. 
Then $\mathcal{A}_1,\ldots ,\mathcal{A}_n$ are infinitesimally free of order $k$ in 
$(\mathcal{A},\varphi ,\varphi ^{(1)},\ldots ,\varphi ^{(k)})$.
\end{Cor}

\begin{center} 
{\large\bf 4. Addition and multiplication of infinitesimally free random variables}
\end{center}

In this section, we consider $n$-tuples of 
infinitesimal noncommutative random variables 
$(a_1,\ldots ,a_n),(b_1,\ldots ,b_n)\in \mathcal{A}^n$ 
(where $(\mathcal{A}, (\varphi ^{(i)})_{0\leq i\leq k})$ 
is an infinitesimal noncommutative probability space of order $k$), 
with respective infinitesimal distributions $(\mu ^{(i)})_{0\leq i\leq k}$ 
and $(\nu ^{(i)})_{0\leq i\leq k}$. We assume that the sets 
$\{a_1,\ldots ,a_n\}$ and $\{b_1,\ldots ,b_n\}$ are infinitesimally free of order $k$ and 
we are interested in the distributions 
of the sum $(a_1,\ldots ,a_n)+(b_1,\ldots ,b_n)$ and of the product $(a_1b_1,\ldots ,a_nb_n)$. 

\begin{center}
{\bf 4.1. Addition of infinitesimally free random variables}
\end{center}

We do not provide a proof of the following result, 
which is a straightforward calculation using 
multilinearity of the infinitesimal cumulant functionals 
and definition of infinitesimal freeness. 
 
\begin{Prop} \label{propinfadd}
Let $(\mathcal{A}, (\varphi ^{(i)})_{0\leq i\leq k})$ be an 
infinitesimal noncommutative probability space of order $k$. 
Consider subsets $\mathcal{M}_1,\mathcal{M}_2$ of $\mathcal{A}$ 
that are infinitesimally free of order $k$. 
Then, one has, for each $n\geq 1$, 
each $n$-tuples $(a_1,\ldots ,a_n)\in \mathcal{M}_1^n,(b_1,\ldots ,b_n)\in \mathcal{M}_2^n$ 
and each $0\leq i\leq k$ : 
\begin{equation} \label{infadd}
\kappa _n^{(i)}(a_1+b_1,\ldots ,a_n+b_n)=\kappa _n^{(i)}(a_1,\ldots ,a_n)+\kappa _n^{(i)}(b_1,\ldots ,b_n).
\end{equation}
\end{Prop}

Using formulas \eqref{inffmc} and \eqref{infinvfmc}, 
the quantities $\kappa _m^{(i)}(a_{i_1},\ldots ,a_{i_m})$, 
$\kappa _m^{(i)}(b_{j_1},\ldots ,b_{j_m})$ for each $0\leq i\leq k$, each $m\geq 1$ 
and each subsets $\{i_1,\ldots ,i_m\} , \{j_1,\ldots ,j_m\}\subseteq [n]$ 
called respectively infinitesimal cumulants of $(a_1,\ldots ,a_n)$ and $(b_1,\ldots ,b_n)$ 
completely determine and are completely determined by the 
infinitesimal distributions of $(a_1,\ldots ,a_n)$ and $(b_1,\ldots ,b_n)$.  
Proposition \ref{propinfadd} thus has the following corollary. 

\begin{Cor} 
Let $(\mathcal{A}, (\varphi ^{(i)})_{0\leq i\leq k})$ 
be an infinitesimal noncommutative probability space of order $k$, 
and $(a_1,\ldots ,a_n),(b_1,\ldots ,b_n)\in \mathcal{A}^n$ with respective infinitesimal distributions 
$(\mu ^{(i)})_{0\leq i\leq k}$ and $(\nu ^{(i)})_{0\leq i\leq k}$. 
If the sets $\{a_1,\ldots ,a_n\}$ and $\{b_1,\ldots ,b_n\}$ are infinitesimally free of order $k$, 
then the infinitesimal distribution of $(a_1,\ldots ,a_n)+(b_1,\ldots ,b_n)$ only depends on 
$(\mu ^{(i)})_{0\leq i\leq k}$ and $(\nu ^{(i)})_{0\leq i\leq k}$. 
It is called the infinitesimal free additive convolution of order $k$
of $(\mu ^{(i)})_{0\leq i\leq k}$ and $(\nu ^{(i)})_{0\leq i\leq k}$ 
and denoted by $(\mu ^{(i)})_{0\leq i\leq k}\boxplus ^{(k)} (\nu ^{(i)})_{0\leq i\leq k}$.
\end{Cor}
 
The corollary above means that the infinitesimal free additive convolution of order $k$ 
defines an operation on infinitesimal laws. The practical way to 
compute the infinitesimal free additive convolution of order $k$ of two infinitesimal laws 
is to use consecutively the inverse of the infinitesimal version 
of the free moment-cumulant formula (formula \eqref{infinvfmc}), 
the additivity of infinitesimal cumulants (formula \eqref{infadd}), 
and finally the infinitesimal version 
of the free moment-cumulant formula (formula \eqref{inffmc}). 
One may find easier to make the computations in a 
scarce $\mathcal{C}_k$-noncommutative probability space. 
\\
Taking into account the link \eqref{decompcum} between infinitesimal cumulant functionals and 
$\mathcal{C}_k$-non-crossing cumulant functionals, Proposition \ref{propinfadd} admits the following corollaries :  

\begin{Cor} \label{propadd}
Let $( \mathcal{A} , \tilde \varphi )$ be a scarce $\mathcal{C}_k$-noncommutative probability space. 
Consider subsets $\mathcal{M}_1,\mathcal{M}_2$ of $\mathcal{A}$ 
that are infinitesimally free of order $k$.
Then, one has, for each $n\geq 1$ and each $n$-tuples 
$(a_1,\ldots ,a_n)\in \mathcal{M}_1^n,(b_1,\ldots ,b_n)\in \mathcal{M}_2^n$ 
$$\tilde \kappa _n(a_1+b_1,\ldots ,a_n+b_n)=\tilde \kappa _n(a_1,\ldots ,a_n)+\tilde \kappa _n(b_1,\ldots ,b_n).$$
\end{Cor}

\begin{Cor} 
Let $( \mathcal{A} , \tilde \varphi )$ be a scarce $\mathcal{C}_k$-noncommutative probability space. 
Consider $a,b\in \mathcal{A} $ that are infinitesimally free of order $k$, then 
$$\tilde R_{a+b}=\tilde R_a+\tilde R_b.$$
\end{Cor}

\begin{Rem} 
Using Corollary \ref{propadd}, it is possible to state and prove $\mathcal{C}_k$-valued 
versions of some famous limit theorems of free probability. 
We discuss this without going into the details ; for a more complete discussion 
of limit theorems in free probability of type B, we refer to \cite{pop} and \cite{bs}.
In a scarce $\mathcal{C}_k$-noncommutative probability space $( \mathcal{A} , \tilde \varphi )$, 
consider a sequence $(a_n)_{n\in \NN}\in \mathcal{A} ^\NN$ 
of centered infinitesimally free identically distributed 
$\mathcal{C}_k$-valued noncommutative random variables. Then the moments of the (rescaled 
by a $\frac{1}{\sqrt{N}}$ factor) sum $\frac{1}{\sqrt{N}} \sum _{n=1}^N a_n$ converge 
to a $\mathcal{C}_k$-valued distribution characterized by the vanishing of all of its 
cumulants except the second one : this is the $\mathcal{C}_k$-valued 
version of the free central limit theorem. The distributions that appear 
as limits in the preceding result deserve to be named 
$\mathcal{C}_k$-valued semicircular elements. Their moments may be computed using the 
$\mathcal{C}_k$-valued free moment-cumulant formula. 
Paralelly, a $\mathcal{C}_k$-valued version of the free Poisson theorem 
may also be stated and proved, and thus a $\mathcal{C}_k$-valued 
Poisson distribution may be defined. 
\end{Rem}

\begin{center}
{\bf 4.2 Multiplication of infinitesimally free random variables}
\end{center}

We now investigate the distribution of the product of 
$n$-tuples of noncommutative random variables that are infinitesimally free of order $k$. 
We first focus on a $\mathcal{C}_k$-noncommutative probability space because, 
the combinatorics being the same in this setting as in usual free probability, 
the proofs and results will be straightforward adaptations of the usual ones, 
which can be found in \cite{nsbook} for instance.

\begin{Prop} 
Let $( \mathcal{A} , \tilde \varphi )$ be a scarce $\mathcal{C}_k$-noncommutative probability space. 
Consider subsets $\mathcal{M}_1,\mathcal{M}_2$ of $\mathcal{A}$ 
that are infinitesimally free of order $k$. 
Then, one has, for each $n\geq 1$ and each $n$-tuples 
$(a_1,\ldots ,a_n)\in \mathcal{M}_1^n,(b_1,\ldots ,b_n)\in \mathcal{M}_2^n$ 
\begin{equation} \label{altprod}
\tilde \kappa _n(a_1b_1,\ldots ,a_nb_n)=\sum_{p\in NC^{(A)}(n)} 
\tilde \kappa _p(a_1,\ldots ,a_n) \tilde \kappa _{\mbox{Kr}(p)}(b_1,\ldots ,b_n).
\end{equation}
\end{Prop}

\begin{proof}
Using Proposition \ref{cwpae}, the left-hand side of \eqref{altprod} is equal to 
$$\sum_{ \begin{array}{c}
{\scriptstyle \pi \in NC(2n) \ such} \\
{\scriptstyle that \ \pi \vee \theta = 1_s}
\end{array} }\tilde \kappa _\pi (a_1,b_1,a_2,\ldots ,b_{n-1},a_n,b_n),$$ 
where $\theta $ is the partition $\{\{1,2\},\ldots ,\{2n-1,2n\}\}$.\\
By the vanishing of mixed cumulants condition, 
the only contributing terms are those indexed by 
non-crossing partitions $\pi $ which are reunion 
of a non-crossing partition $p$ of $\{1,3,\ldots ,2n-1\}$ 
and a non-crossing partition $q$ of $\{2,4,\ldots ,2n\}$. 
The condition $\pi \vee \theta = 1_s$ for such a partition $\pi $ 
may be reinterpreted as $q=\mbox{Kr}(p)$ 
(up to the identifications $\{1,3,\ldots ,2n-1\}\leftrightarrow [n]$ 
and $\{2,4,\ldots ,2n\}\leftrightarrow \overline{[n]}$).  
\end{proof}

Switching to the infinitesimal framework, one can state the 
following result. 

\begin{Cor} 
Let $(\mathcal{A}, (\varphi ^{(i)})_{0\leq i\leq k})$ 
be an infinitesimal noncommutative probability space of order $k$, 
and $(a_1,\ldots ,a_n),(b_1,\ldots ,b_n)\in \mathcal{A}^n$ be $n$-tuples 
with respective infinitesimal distributions 
$(\mu ^{(i)})_{0\leq i\leq k}$ and $(\nu ^{(i)})_{0\leq i\leq k}$.\\
If the sets $\{a_1,\ldots ,a_n\}$ and $\{b_1,\ldots ,b_n\}$ are infinitesimally free of order $k$, 
then the infinitesimal distribution of $(a_1b_1,\ldots ,a_nb_n)$ only depends on 
$(\mu ^{(i)})_{0\leq i\leq k}$ and $(\nu ^{(i)})_{0\leq i\leq k}$. 
It is denoted by $(\mu ^{(i)})_{0\leq i\leq k}\boxtimes ^{(k)} (\nu ^{(i)})_{0\leq i\leq k}$
and called the infinitesimal free multiplicative convolution of order $k$
of $(\mu ^{(i)})_{0\leq i\leq k}$ and $(\nu ^{(i)})_{0\leq i\leq k}$.
\end{Cor}

If $a,b\in \mathcal{A}$ are $\mathcal{C}_k$-noncommutative random variables 
that are infinitesimally free of order $k$ in a scarce $\mathcal{C}_k$-noncommutative 
probability space, the $\mathcal{C}_k$-valued R-transform of $a\cdot b$ is 
$\tilde R_{a\cdot b}=\tilde R_{a}$\framebox[7pt]{$\star$}$_{\mathcal{C}_k}\tilde R_{b}$, 
where \framebox[7pt]{$\star$}$_{\mathcal{C}_k}$ is the version of the boxed convolution operation 
introduced in \cite{ns96} with scalars in $\mathcal{C}_k$. We recall in the next subsection 
the definition and main properties of this operation. 

\begin{center}
{\bf 4.3 Boxed convolution of type A}
\end{center}

An operation on formal power series in several noncommuting indeterminates and with complex coefficients 
is introduced in \cite{ns96}, and called boxed convolution. It is defined as a convolution on the 
lattices of non-crossing partitions (of type A). We recall here this definition, but for series 
in only one variable (for simplicity) and with coefficients in any unital complex algebra. 
This is already the point of view adopted in \cite{bgn}. 

\begin{Def} \label{starbox}
Let $\mathcal{C}$ be a unital commutative algebra over $\CC$.
On $\Theta _{\mathcal{C}}^{(A)}$ we define a binary operation 
\framebox[7pt]{$\star$}$_{\mathcal{C}}^{(A)}$, as follows. 
If $$f(z)=\sum_{n=1}^\infty  \alpha _nz^n \in \Theta _{\mathcal{C}}^{(A)},$$ and 
$$g(z)=\sum_{n=1}^\infty \beta _nz^n \in \Theta _{\mathcal{C}}^{(A)},$$
then $f$\framebox[7pt]{$\star$}$_{\mathcal{C}}^{(A)}g$ is the series $\sum_{n=1}^\infty  \gamma _nz^n$, where 
$$\gamma _m=\sum_{\substack{p\in NC^{(A)}(m)\\p:=\{E_1,\ldots ,E_h\}\\\mbox{Kr}(p):=\{F_1,\ldots ,F_l\}}} 
(\prod_{i=1}^h\alpha _{\mbox{card}(E_i)}) \cdot (\prod_{j=1}^l\beta _{\mbox{card}(F_j)}).$$
\end{Def}

\begin{Rem} \label{multi}
It is obviously possible to define a boxed convolution operation for power series in several 
noncommuting indeterminates and with coefficients in $\mathcal{C}$. 
The formulas are the same as in the case of complex coefficients, which first appeared 
in \cite{ns96} and can also be found in \cite{nsbook}.
\end{Rem}

The operation \framebox[7pt]{$\star$}$_{\mathcal{C}}^{(A)}$ is associative, 
commutative and the series $$\Delta _\mathcal{C}^{(A)}(z):=1_{\mathcal{C}}z$$ is its unit element. 
There is another important series in $\Theta _{\mathcal{C}}^{(A)}$, 
namely $$\zeta _\mathcal{C}^{(A)}(z):=\sum_{n=1}^\infty 1_{\mathcal{C}}z^n.$$
Notice that a series $f\in \Theta _{\mathcal{C}}^{(A)}$ is invertible with respect to 
\framebox[7pt]{$\star$}$_{\mathcal{C}}^{(A)}$ if and only if its coefficient of degree one is 
itself invertible in the algebra $\mathcal{C}$. 
In particular, $\zeta _\mathcal{C}^{(A)}$ is invertible 
with respect to \framebox[7pt]{$\star$}$_{\mathcal{C}}^{(A)}$, 
and its inverse is called the M\"oebius series, and denoted by $\mbox{M\"ob} _\mathcal{C}^{(A)}$. 
The proofs of these claims may be obtained
by a straightforward adaptation of the proofs given in \cite{ns96}. 
The free moment-cumulant relation of free probability 
(and its inverse) may be read at the level of power series : 
more precisely, in a noncommutative probability space $(\mathcal{A} , \varphi )$, 
the moment series and the R-transform of $a\in \mathcal{A} $ 
satisfy the following relations: 
$M_a=R_a$\framebox[7pt]{$\star$}$_{\CC}^{(A)} \zeta _{\CC}^{(A)}$,
$R_a=M_a$\framebox[7pt]{$\star$}$_{\CC}^{(A)} \mbox{M\"ob} _{\CC}^{(A)}$.
These formulas have infinitesimal analogues, as stated in the next proposition. 
It is indeed straightforward to check that, 
in the particular case of single variables, 
the formulas \eqref{fmc} and \eqref{invfmc} may be read 
at the level of power series as follows : 

\begin{Prop} 
Let $( \mathcal{A} , \tilde \varphi )$ be a scarce $\mathcal{C}_k$-noncommutative probability space 
and consider a $\mathcal{C}_k$-noncommutative random variable $a\in \mathcal{A}$.
Then the $\mathcal{C}_k$-valued moment series $\tilde M_a$ and 
the $\mathcal{C}_k$-valued R-transform $\tilde R_a$ of $a$ are related by the equivalent formulas : 
$\tilde M_a=\tilde R_a$\framebox[7pt]{$\star$}$_{\mathcal{C}_k}^{(A)} \zeta _{\mathcal{C}_k}^{(A)}$,
$\tilde R_a=\tilde M_a$\framebox[7pt]{$\star$}$_{\mathcal{C}_k}^{(A)} \mbox{M\"ob} _{\mathcal{C}_k}^{(A)}.$
\end{Prop}

The importance of boxed convolution (with complex coefficients) 
in free probability also comes from the fact, 
proved in \cite{ns96}, that \framebox[7pt]{$\star$}$_{\CC}^{(A)}$ 
provides the combinatorial description 
for the multiplication of two free noncommutative random variables, 
in terms of their R-transforms.
More precisely, we have, for free $a,b$ in a 
noncommutative probability space $(\mathcal{A} , \varphi )$ :
$R_{a\cdot b}=R_a$\framebox[7pt]{$\star$}$_{\CC}^{(A)}R_b$.

Interestingly, \framebox[7pt]{$\star$}$_{\mathcal{C}_k}^{(A)}$ also 
provides the combinatorial description 
for the multiplication of two infinitesimally free infinitesimal
noncommutative random variables, in terms of their R-transforms. 

\begin{Prop} 
Let $( \mathcal{A} , \tilde \varphi )$ be a scarce $\mathcal{C}_k$-noncommutative probability space. 
Consider $a,b\in \mathcal{A} $ that are infinitesimally free of order $k$, then
\center{$\tilde R_{a\cdot b} = \tilde R_a$ \framebox[7pt]{$\star$}$_{\mathcal{C}_k}^{(A)}\tilde R_b$.}
\end{Prop}

In \cite{ns97}, a "Fourier transform" is introduced for multiplicative functions 
on non-crossing partitions. It is barely a map $\mathcal{F}$ which associates to 
$f(z)=\sum_{n=1}^\infty  \alpha _nz^n\in \Theta _{\CC}^{(A)}$, with $\alpha _1\not=0$ 
(to ensure that $f$ is invertible with respect to the composition of formal power series), 
the series $\mathcal{F}(f)(z):=\frac{1}{z}f^{\langle -1\rangle}(z)$. 
The map $\mathcal{F}$ has the important property to convert the boxed convolution into 
the multiplication of formal power series : 
$\mathcal{F}(f$\framebox[7pt]{$\star$}$_{\CC}^{(A)}g)=\mathcal{F}(f)\cdot \mathcal{F}(g)$. 
If $a$ is a noncommutative random variable with non-zero mean and R-transform $R_a$ 
in a noncommutative probability space, the series $\mathcal{F}(R_a)$ is of great importance : 
this is a combinatorial approach to Voiculescu's S-tranform \cite{voi2}. 
As noticed in \cite{pop}, the combinatorial proofs remain valid for series with 
$\mathcal{C}_k$-valued coefficients such that the coefficient of degree one is invertible. 

\begin{Def} 
Let $( \mathcal{A} , \tilde \varphi )$ be a scarce $\mathcal{C}_k$-noncommutative probability space. 
The $\mathcal{C}_k$-valued S-transform of an infinitesimal noncommutative 
random variable $a\in \mathcal{A} $ such that $\tilde \varphi (a)$ is invertible in $\mathcal{C}_k$
is the power series $\tilde S_a\in \Theta ^{(k)}$ defined as follows:
$$\tilde S_a(z):=\frac{1}{z}\tilde R_a^{\langle -1\rangle}(z).$$
\end{Def}

\begin{Prop} 
Let $( \mathcal{A} , \tilde \varphi )$ be a scarce $\mathcal{C}_k$-noncommutative probability space. 
Consider $a,b\in \mathcal{A} $ that are infinitesimally free of order $k$, and such that 
$\tilde \varphi (a)$ and $\tilde \varphi (b)$ are invertible in $\mathcal{C}_k$, then
the $\mathcal{C}_k$-valued S-transform $\tilde S_{a\cdot b}$ of $a\cdot b$ satisfies: 
$$\tilde S_{a\cdot b}(z)=\tilde S_a(z) \tilde S_b(z).$$
\end{Prop}
\

Practically speaking, the computation of the distribution of 
the product of two infinitesimally free infinitesimal noncommutative 
random variables requires a good understanding of the $\mathcal{C}_k$-valued 
version of the boxed convolution. 
More precisely, in the notations of Definition \ref{starbox}, 
it would be of interest to have a formula for $\gamma _m^{(i)}$ 
as a function of the $\alpha _n^{(j)}$'s and the $\beta _n^{(j)}$'s.\\ 
As mentioned before, the version of the boxed convolution with scalars in $\mathcal{C}_0=\CC$ 
is a classical operation in free probability. 
The version of the boxed convolution with scalars in $\mathcal{C}_1=\mathbb{G}$ has 
already been considered in \cite{bgn}, 
where it is shown to coincide with the boxed 
convolution based on non-crossing partitions of type B, 
in connection with free probability of type B. 
This leads to the natural question : does the operation 
\framebox[7pt]{$\star$}$_{\mathcal{C}_k}^{(A)}$ coincide with a boxed convolution 
based on a certain set of special non-crossing partitions. 
In Section 7, we will give a positive answer to this problem, 
by introducing the non-crossing partitions of type $k$.
Before that, we review some background on non-crossing partitions and boxed convolution
of type B.

\begin{center} 
{\bf\large 5. Non-crossing partitions and boxed convolution of type B}
\end{center}

This section is devoted to some background on non-crossing partitions 
of type B and on the boxed convolution of type B. 

\begin{center}
{\bf 5.1 Non-crossing partitions of type B}
\end{center}

As recalled in Section 2, there is a close link between the lattice of non-crossing 
partitions and the Cayley graph of the symmetric group. 
Actually, one may interpret the lattices of non-crossing partitions 
in terms of the root systems of type A, 
justifying the notation $NC^{(A)}(n)$. 
This led Reiner to introduce  in \cite{rei} 
the type B analogue $NC^{(B)}(n)$ of the lattice of non-crossing partitions. 
To this aim, consider the totally ordered set $$[\pm n]=\{1<2<\ldots <n<-1<-2<\ldots <-n\}.$$
One defines $NC^{(B)}(n)$ to be the subset of $NC^{(A)}([\pm n])$ 
consisting of non-crossing partitions that are invariant under the inversion map $x\mapsto -x$.\\
In such a partition $\pi \in NC^{(B)}(n)$, there is at most one block that is inversion-invariant, 
called, whenever it exists, the zero-block of $\pi $.
The other blocks of $\pi $ come two by two: if $F$ is a block which is not inversion-invariant, 
then $-F$ is another block (obviously not inversion-invariant).\\
It is immediate that $NC^{(B)}(n)$ is a sublattice of $NC^{(A)}([\pm n])$, 
with the same minimal and maximal elements.\\
Moreover, $NC^{(B)}(n)$ is closed under the Kreweras complementation maps 
$\mbox{Kr}$ and $\mbox{Kr}'$ (considered on $NC^{(A)}([\pm n])$).\\
When restricted from $NC^{(A)}([\pm n])$ to $NC^{(B)}(n)$, 
these maps will then give two anti-isomorphisms of $NC^{(B)}(n)$, 
inverse to each other, and which will also be called (without ambiguity) 
Kreweras complementation maps (on $NC^{(B)}(n)$).
In this case, the important relation \eqref{nbblocks} becomes
$$| \pi |+| \mbox{Kr}(\pi ) |=2n+1, \forall \pi \in NC^{(B)}(n).$$
As a consequence, for $\pi \in NC^{(B)}(n)$, 
exactly one of the two partitions $\pi $ and $\mbox{Kr}(\pi )$ has a zero-block. 
In the description of a non-crossing partition of type B, a role is played by the absolute value map 
$\mbox{Abs} : [\pm n] \longrightarrow [n]$ sending $\pm i$ to $i$.

\begin{Not} 
Any map $f$ defined from $[m]$ into $[n]$ is naturally extended 
to a map from $[m]\cup \overline{[m]}$ into $[n]\cup \overline{[n]}$ 
by simply requiring that $f(\overline{i})=\overline{f(i)}$.\\
Moreover, if $Y$ is a subset of $[m]\cup \overline{[m]}$, 
we will use the notation $f(Y)$ for the set 
$\{f(y), y\in Y\} \subset [n]\cup \overline{[n]}$.\\
Finally, given a collection $\Upsilon $ of subsets of $[m]\cup \overline{[m]}$, 
we will denote by $f(\Upsilon )$ the collection 
$\{f(Y), Y\in \Upsilon \}$ of subsets of $[n]\cup \overline{[n]}$.
\end{Not}

Let us now state a key result of \cite{bgn}.

\begin{Theo} 
$\pi \mapsto \mbox{Abs}(\pi )$ is a $(n+1)$-to-$1$ map from $NC^{(B)}(n)$ onto $NC^{(A)}(n)$.
\end{Theo}

We refer to the paper \cite{bgn} for the proof. In the next subsection, we recall the definition 
of the type B analogue of the boxed convolution operation and give the announced result stating 
that this operation is a boxed convolution of type A on the algebra $\mathcal{C}_1$.

\begin{center}
{\bf 5.2 Boxed convolution of type B}
\end{center}

\begin{Def} 
\begin{enumerate}
\item We denote by $\Theta ^{(B)}$ the set of power series of the form 
$$f(z)=\sum_{n=1}^\infty  (\alpha _n',\alpha _n'')z^n,$$ 
where the $\alpha _n'$'s and $\alpha _n''$'s are complex numbers.
\item Let $f(z):=\sum_{n=1}^\infty (\alpha _n',\alpha _n'')z^n$ and 
$g(z)=\sum_{n=1}^\infty (\beta _n',\beta _n'')z^n$ be in $\Theta ^{(B)}$.
For every $m\geq 1$, consider the numbers $\gamma _m'$ and $\gamma _m''$ defined by
$$\gamma _m'=\sum_{\substack{p\in NC^{(A)}(m)\\p:=\{E_1,\ldots ,E_h\}\\\mbox{Kr}(p):=\{F_1,\ldots ,F_l\}}} 
(\prod_{i=1}^h\alpha _{\mbox{card}(E_i)}') \cdot (\prod_{j=1}^l\beta _{\mbox{card}(F_j)}'),$$
$$\gamma _m''=\sum_{\substack{p\in NC^{(B)}(m)\,with\,zero-block\\p:=\{Z,X_1,-X_1,\ldots ,X_h,-X_h\}\\
\mbox{Kr}(p):=\{Y_1,-Y_1,\ldots ,Y_l,-Y_l\}}} (\prod_{i=1}^h\alpha _{\mbox{card}(X_i)}') \cdot 
\alpha _{\mbox{card}(Z)/2}'' \cdot (\prod_{j=1}^l\beta _{\mbox{card}(Y_j)}')$$
$$+ \sum_{\substack{p\in NC^{(B)}(m)\,without\,zero-block\\p:=\{X_1,-X_1,\ldots ,X_h,-X_h\}\\\mbox{Kr}(p)
:=\{Z,Y_1,-Y_1,\ldots ,Y_l,-Y_l\}}} (\prod_{i=1}^h\alpha _{\mbox{card}(X_i)}') \cdot \beta _{\mbox{card}(Z)/2}'' 
\cdot (\prod_{j=1}^l\beta _{\mbox{card}(Y_j)}').$$
Then the series $\sum_{n=1}^\infty  (\gamma _n',\gamma _n'')z^n$ is called the boxed convolution of type B of $f$ and $g$, 
and is denoted by $f$ \framebox[7pt]{$\star$}$^{(B)} g$.
\end{enumerate}
\end{Def}

\begin{Theo} \cite{bgn} Theorem 5.3
\framebox[7pt]{$\star$}$^{(B)}=$\framebox[7pt]{$\star$}$_{\mathcal{C}_1}^{(A)}$
\end{Theo}

We now introduce the non-crossing partitions of type $k$, 
generalizing non-crossing partitions of type A and B. 

\begin{center} 
{\bf\large 6. Non-crossing partitions of type $k$}
\end{center}

This section is devoted to the introduction and study of a 
set of non-crossing partitions, namely the set of non-crossing partitions of type $k$, 
which has to be a cover of $NC{(A)}(n)$ 
related to the version of the boxed convolution with scalars 
in $\mathcal{C}_k$.

\begin{center}
{\bf 6.1 Definition and first properties}
\end{center}

\begin{Def} 
Let $n$ be a positive integer.
We call {\em reduction mod $n$ map} the map $$\mbox{Red}_n^{(k)} : [(k+1)n] \rightarrow [n]$$
sending each $i\in [(k+1)n]$ to its congruence class mod $n$.
\end{Def}

\begin{Rem} 
For $k=0$, the map $\mbox{Red}_n^{(0)}$ is simply the identity map on $[n]$.\\
For $k=1$, up to identifying $[2n]$ with $[\pm n]$, the map $\mbox{Red}_n^{(1)}$ 
is identified with $\mbox{Abs}$.
\end{Rem}

\begin{Def} \label{redprop}
A non-crossing partition $\pi $ of $[(k+1)n]$ is said to satisfy {\em the mod $n$ reduction property} if 
$\mbox{Red}_n^{(k)}(\pi )$ is a non-crossing partition of $[n]$ and if 
$\mbox{Red}_n^{(k)}(\mbox{Kr}(\pi ))$ is a non-crossing partition of $\overline{[n]}$.
\end{Def}

{\em Non-crossing partitions of type $k$} are the non-crossing partitions of $[(k+1)n]$ 
satisfying the mod $n$ reduction property. 

\begin{Def} 
We write $NC^{(k)}(n)$ for the set of non-crossing partitions of type $k$, 
that is non-crossing partitions of $[(k+1)n]$ satisfying the mod $n$ reduction property.
\end{Def}

\begin{Rem} 
All non-crossing partitions of $[n]$ trivially satisfy the mod $n$ reduction property 
(since $\mbox{Red}_n^{(0)}$ is simply the identity map).
Hence $NC^{(0)}(n)=NC^{(A)}(n)$.
\end{Rem}

The next proposition states that the non-crossing partitions of type $k$ are a 
generalization of the non-crossing partitions of type B.

\begin{Prop} 
If we identify $[\pm n]$ with $[2n]$ and $\mbox{Abs}$ with $\mbox{Red}_n^{(1)}$,\\
then $NC^{(B)}(n)=NC^{(1)}(n)$.
\end{Prop}

\begin{proof}
That $\pi \in NC^{(B)}(n)$ satisfies the mod $n$ reduction property 
is a corollary of Proposition 1.3 and Lemma 1 in \cite{bgn}.\\
For the converse, let $\pi \in NC^{(1)}(n)$ satisfy the mod $n$ reduction property, 
and assume that there exist two elements $x,y\in [\pm n]$ 
such that $x\sim _\pi y$ , $-x\not\sim _\pi -y$. 
By reduction mod $n$ property, we necessarily have $-y\sim _\pi x\sim _\pi y\sim _\pi -x$, which is a contradiction. 
\end{proof}

\begin{Rem} \label{B}
The proof above and Lemma 1 in \cite{bgn} show that, for a non-crossing partition $\pi $ 
of $[2n]$, the mod $n$ reduction property is equivalent to the only requirement that $\mbox{Red}_n^{(1)}(\pi )$ 
is a non-crossing partition of $[n]$.
\end{Rem}

In Definition \ref{redprop}, the reduction mod $n$ property for a non-crossing partition $\pi $ of $[(k+1)n]$
consists of two requirements : $\mbox{Red}_n^{(k)}(\pi )$ has to be a non-crossing partition of $[n]$ 
and $\mbox{Red}_n^{(k)}(\mbox{Kr}(\pi ))$ has to be a non-crossing partition of $\overline{[n]}$. 
Actually, there is a slightly stronger characterization stated in the next proposition.

\begin{Prop} \label{charac}
A non-crossing partition $\pi $ of $[(k+1)n]$ satisfies the reduction mod $n$ property if and only if 
$\mbox{Red}_n^{(k)}(\pi \cup \mbox{Kr}(\pi ))$ is a non-crossing partition of $[n]\cup \overline {[n]}$.
\end{Prop}

\begin{proof}
If $\mbox{Red}_n^{(k)}(\pi \cup \mbox{Kr}(\pi ))$ is a non-crossing partition of $[n]\cup \overline {[n]}$, 
since $\mbox{Red}_n^{(k)}(\pi )$ is a family of subsets of $[n]$ and $\mbox{Red}_n^{(k)}(\mbox{Kr}(\pi ))$ 
is a family of subsets of $\overline{[n]}$, they have to be non-crossing partitions of $[n]$ and $\overline{[n]}$ 
respectively ; in other words $\pi $ has to satisfy the reduction mod $n$ property.\\
We assume now that $\pi $ is a non-crossing partition of $[(k+1)n]$ satisfying the reduction mod $n$ property, 
and aim at proving that $\mbox{Red}_n^{(k)}(\pi \cup \mbox{Kr}(\pi ))$ 
is a non-crossing partition of $[n]\cup \overline {[n]}$.\\
By the reduction property, $\mbox{Red}_n^{(k)}(\pi \cup \mbox{Kr}(\pi ))=
\mbox{Red}_n^{(k)}(\pi )\cup \mbox{Red}_n^{(k)}(\mbox{Kr}(\pi ))$ 
is the union of a partition of $[n]$ and of a partition of $\overline{[n]}$, 
and hence a partition of $[n]\cup \overline {[n]}$. 
To prove that this partition is non-crossing, consider four elements 
$a<b<c<d$ of $[n]\cup \overline {[n]}$, such that 
$a\sim _{\mbox{Red}_n^{(k)}(\pi \cup \mbox{Kr}(\pi ))}c$ and 
$b\sim _{\mbox{Red}_n^{(k)}(\pi \cup \mbox{Kr}(\pi ))}d$. 
We have to show that $a\sim _{\mbox{Red}_n^{(k)}(\pi \cup \mbox{Kr}(\pi ))}b$.\\
Let $1\leq i_0\leq (k+1)n+1$ be minimal with the property that $\mbox{Mix}(\pi ,i_0)$ contains 
an element $x$ such that $\mbox{Red}_n^{(k)}(x)\in \{a,b,c,d\}$. 
Choose also the smallest such $x$. 
We may assume that $\mbox{Red}_n^{(k)}(x)=a$ (the other cases are similar). 
By assumption, $c\in \mbox{Red}_n^{(k)}(\mbox{Mix}(\pi ,i_0))$ : 
there is an element $z\in \mbox{Mix}(\pi ,i_0)$ such that $\mbox{Red}_n^{(k)}(z)=c$. 
Our choice of $x$ ensures that $x<z$ and there is 
necessarily an element $x<y<z$ such that $\mbox{Red}_n^{(k)}(y)=b$.
By minimality of $i_0$, $y\in \mbox{Mix}(\pi ,i_0)$, hence $b\in \mbox{Red}_n^{(k)}(\mbox{Mix}(\pi ,i_0))$
is linked to $a$ in ${\mbox{Red}_n^{(k)}(\pi \cup \mbox{Kr}(\pi ))}$ and we are done.
\end{proof}

\begin{Rem} 
When $k=0$, the reduction mod $n$ property is 
satisfied by any non-crossing partition of $[n]$ 
and is in particular equivalent to the only empty requirement : 
$\pi \in NC^{(A)}(n)$ satisfies the reduction mod $n$ property if and 
only if $\mbox{Red}_n^{(0)}(\pi )$ is a non-crossing partition of $[n]$.\\
As explained in Remark \ref{B}, this is also the case when $k=1$ : 
$\pi \in NC^{(A)}(2n)$ satisfies the reduction mod $n$ property if and 
only if $\mbox{Red}_n^{(1)}(\pi )$ is a non-crossing partition of $[n]$.\\ 
Assume now that $k\geq 2$ ; the situation then is different.\\
As an example, for $k=2$ and $n=2$, consider the partition 
$$\pi :=\{\{1,2,3\},\{4,5,6\}\}\in NC^{(A)}(6).$$ 
It is straightforward to check that $\mbox{Red}_2^{(2)}(\pi )=\{1,2\}$ is a non-crossing partition 
of $[2]$. However, from the easy computation 
$\mbox{Kr}(\pi )=\{\{\overline{1}\},\{\overline{2}\},\{\overline{4}\},\{\overline{5}\},\{\overline{3},\overline{6}\}\}$, 
we deduce that $\mbox{Red}_2^{(2)}(\mbox{Kr}(\pi ))$ is not a partition of $\overline{[6]}$ and consequently 
that $\pi $ does not satisfy the reduction mod $2$ property.
\end{Rem}

The following proposition states that the Kreweras complementation maps 
may be considered as two order-reversing bijections of $NC^{(k)}(n)$.

\begin{Prop} 
The restrictions from $NC^{(A)}((k+1)n)$ to $NC^{(k)}(n)$ of $\mbox{Kr}$ and $\mbox{Kr}'$ 
are order-reversing bijections of $NC^{(k)}(n)$.
\end{Prop}

The name of Kreweras complementation map and the notations $\mbox{Kr}$, $\mbox{Kr}'$ 
will be conserved as there is no ambiguity about the meaning of $\mbox{Kr}(\pi )$ or $\mbox{Kr}'(\pi )$ 
whether $\pi $ is viewed as an element of $NC^{(k)}(n)$ or of $NC^{(A)}((k+1)n)$.

\begin{proof}
It is clearly sufficient to prove that the non-crossing partition $\mbox{Kr}(\pi )$ 
of $\overline{[(k+1)n]}$ satisfies the reduction mod $n$ property whenever $\pi $ does. 
Assume that the non-crossing partition $\pi $ of $[(k+1)n]$ satisfies the reduction mod $n$ property. 
By assumption, $\mbox{Red}_n^{(k)}(\mbox{Kr}(\pi ))$ is a non-crossing partition of $\overline{[n]}$. 
It remains to prove that $\mbox{Red}_n^{(k)}(\mbox{Kr}^2(\pi ))$ 
is a non-crossing partition of $\overline{\overline{[n]}}$.\\
From the geometric description of $\mbox{Kr}^2(\pi )$ given in Section 3, 
we deduce that $\mbox{Red}_n^{(k)}(\mbox{Kr}^2(\pi ))$ is obtained 
from $\mbox{Red}_n^{(k)}(\pi )$ by a rotation. 
By reduction mod $n$ property, $\mbox{Red}_n^{(k)}(\pi )$ 
is a non-crossing partition of $[n]$, 
so $\mbox{Red}_n^{(k)}(\mbox{Kr}^2(\pi ))$ is itself 
a non-crossing partition of $\overline{\overline{[n]}}$. 
Thus the proof is complete.
\end{proof}

Given $\pi \in NC^{(k)}(n)$, $\mbox{Kr}(\mbox{Red}_n^{(k)}(\pi ))$ and $\mbox{Red}_n^{(k)}(\mbox{Kr}(\pi ))$ 
are thus two non-crossing partitions of $\overline{[n]}$.
The following lemma, generalizing Lemma 1 of \cite{bgn}, states that these two partitions coincide.

\begin{Prop} \label{comm}
$\forall \pi \in NC^{(k)}(n), \mbox{Kr}(\mbox{Red}_n^{(k)}(\pi ))=\mbox{Red}_n^{(k)}(\mbox{Kr}(\pi ))$.
\end{Prop}

\begin{proof}
Let $\pi$ be a non-crossing partition of type $k$. 
By Proposition \ref{charac}, $\mbox{Red}_n^{(k)}(\pi )\cup \mbox{Red}_n^{(k)}(\mbox{Kr}(\pi ))=
\mbox{Red}_n^{(k)}(\pi \cup \mbox{Kr}(\pi ))$ is a non-crossing partition of $[n]\cup \overline {[n]}$.
Since $\mbox{Kr}(\mbox{Red}_n^{(k)}(\pi ))$ is maximal with the property that 
$\mbox{Red}_n^{(k)}(\pi )\cup \mbox{Kr}(\mbox{Red}_n^{(k)}(\pi ))$ is non-crossing, 
it follows that $$\mbox{Red}_n^{(k)}(\mbox{Kr}(\pi ))\preceq \mbox{Kr}(\mbox{Red}_n^{(k)}(\pi )).$$
There is equality if, for any $\overline x$ having a neighbour 
$\overline y>\overline x$ in $\mbox{Kr}(\mbox{Red}_n^{(k)}(\pi ))$, 
$\overline y$ is linked to $\overline x$ in $\mbox{Red}_n^{(k)}(\mbox{Kr}(\pi ))$.
For such elements $\overline x, \overline y\in \overline{[n]}$, 
we call $V$ the block of $\pi $ containing $x+1$. 
The reduction property implies that $\mbox{Red}_n^{(k)}(V)$ is a block of the partition $\mbox{Red}_n^{(k)}(\pi )$.
By construction of the Kreweras complement, $x+1$ is the smallest element of both $V$ and $\mbox{Red}_n^{(k)}(V)$, 
and $y$ is the greatest element of $\mbox{Red}_n^{(k)}(V)$. 
Consider now the greatest element $z$ of $V$. 
Notice that $x+1\leq \mbox{Red}_n^{(k)}(z)\leq y$. 
By construction of the Kreweras complement again, $\overline x$ is linked to $\overline z$ in $\mbox{Kr}(\pi )$, 
then $\overline x$ is linked to $\mbox{Red}_n^{(k)}(\overline z)=\overline {\mbox{Red}_n^{(k)}(z)}$ 
in $\mbox{Red}_n^{(k)}(\mbox{Kr}(\pi ))$ and therefore in $\mbox{Kr}(\mbox{Red}_n^{(k)}(\pi ))$. 
This means that, if $\mbox{Red}_n^{(k)}(z)<y$, $\overline y$ would not be the neighbour of $\overline x$ in 
$\mbox{Kr}(\mbox{Red}_n^{(k)}(\pi ))$, which is a contradiction.
So $\mbox{Red}_n^{(k)}(z)=y$ or, in other words, 
$\overline x$ is linked to $\overline y$ in $\mbox{Red}_n^{(k)}(\mbox{Kr}(\pi ))$. 
\end{proof}

A deeper description of non-crossing partitions of type $k$ is given in the next subsection.

\begin{center}
{\bf 6.2 Structure of non-crossing partitions of type $k$}
\end{center}

The goal of this subsection is to describe the structure of a non-crossing partition of type $k$. 
In the next proposition, $t$ denotes the bijection between non-crossing partitions and permutations 
lying on a geodesic in the Cayley graph of the symmetric group, introduced by Biane in \cite{bia}, 
and described in Section 3. 
We warn the reader that we choose to use the same notation $t$ for this bijection, 
defined either on $NC(n)$ or $NC((k+1)n)$. We hope that this choice, made in the sake of simplicity, 
will not be a source of confusion in the reader's mind. The content of this proposition is, roughly speaking, 
that a type $k$ non-crossing partition $\pi $ is characterized by the two requirements : 
$\mbox{Red}_n^{(k)}(\pi )$ is a non-crossing partition of $NC^{(A)}(n)$ and 
the elements of each of the blocks of $\pi $ come in the same order as 
their congruence classes in its reduction $\mbox{Red}_n^{(k)}(\pi )$. 

\begin{Prop} \label{biane}
For $\pi \in NC^{(A)}((k+1)n)$ such that $\mbox{Red}_n^{(k)}(\pi )\in NC^{(A)}(n)$, 
$\pi \in NC^{(k)}(n)$ if and only if 
\begin{equation} \label{shift}
\forall x\in [(k+1)n], \mbox{Red}_n^{(k)}(t(\pi )(x))=t(\mbox{Red}_n^{(k)}(\pi ))(\mbox{Red}_n^{(k)}(x)).
\end{equation}
\end{Prop}

\begin{proof}
Assume first that $\pi \in NC^{(k)}(n)$ and fix $x\in [(k+1)n]$. Set $y:=t(\pi )(x)$.\\
By construction of $t$, $y$ is the neighbour of $x$ in $\pi $ 
and $t(\mbox{Red}_n^{(k)}(\pi ))(\mbox{Red}_n^{(k)}(x))$ is the neighbour 
of $\mbox{Red}_n^{(k)}(x)$ in $\mbox{Red}_n^{(k)}(\pi )$. 
By construction of the Kreweras complement, 
$\overline{x}$ is the neighbour of $\overline{y-1}$ in $\mbox{Kr}(\pi )$, 
and $\overline{\mbox{Red}_n^{(k)}(x)}$ is the neighbour of 
$\overline{t(\mbox{Red}_n^{(k)}(\pi ))(\mbox{Red}_n^{(k)}(x))-1}$ 
in $\mbox{Kr}(\mbox{Red}_n^{(k)}(\pi ))=\mbox{Red}_n^{(k)}(\mbox{Kr}(\pi ))$ 
(the latter equality holds because of Proposition \ref{comm}).
By reduction property, $\overline{\mbox{Red}_n^{(k)}(y-1)}$ is linked to $\overline{\mbox{Red}_n^{(k)}(x)}$. 
It follows that the neighbour of $\mbox{Red}_n^{(k)}(x)$ in $\mbox{Red}_n^{(k)}(\pi )$, 
$t(\mbox{Red}_n^{(k)}(\pi ))(\mbox{Red}_n^{(k)}(x))$, is the first point 
coming after $\overline{\mbox{Red}_n^{(k)}(y-1)}$ linked to $\mbox{Red}_n^{(k)}(x)$ : 
it is $\mbox{Red}_n^{(k)}(y)$ and we are done.
For the converse, let $\pi \in NC^{(A)}((k+1)n)$ 
be such that $\mbox{Red}_n^{(k)}(\pi )\in NC^{(A)}(n)$ and 
assume that condition ? holds. 
We have to prove that $\mbox{Red}_n^{(k)}(\mbox{Kr}(\pi ))$ 
is a non-crossing partition of $\overline{[n]}$. 
Let $\overline{x}\in \overline{[(k+1)n]}$, its neighbour in $\mbox{Kr}(\pi )$ 
is $\overline{t(\pi )^{-1}(x+1)}$, by construction of the Kreweras complement. 
It follows of condition ? that 
$\mbox{Red}_n^{(k)}(t(\pi )^{-1}(x+1))=t(\mbox{Red}_n^{(k)}(\pi ))^{-1}(\mbox{Red}_n^{(k)}(x+1))$. 
Hence the congruence class of the neighbour of $\overline{x}$ in $\mbox{Kr}(\pi )$ 
only depends on the congruence class of $x$, and moreover 
$\mbox{Red}_n^{(k)}(\mbox{Kr}(\pi ))=\mbox{Kr}(\mbox{Red}_n^{(k)}(\pi ))$ 
and we are done. 
\end{proof}

The preceding proposition has some important consequences. 

\begin{Cor} 
Let $\pi \in NC^{(A)}((k+1)n)$ and $V$ be a block of $\pi\cup \mbox{Kr}(\pi )$. 
The cardinal of $\mbox{Red}_n^{(k)}(V)$ divides the cardinal of $V$. 
We call {\em multiplicity} of $V$ the quotient 
$$\mbox{mult}_{\pi\cup \mbox{Kr}(\pi )}(V):=\frac{\mbox{card}(V)}{\mbox{card}(\mbox{Red}_n^{(k)}(V))}.$$
This is a positive integer lower or equal than $k+1$.
The blocks of multiplicity $1$ will be called {\em simple}.
\end{Cor}

\begin{proof}
For $x \in V$, the cardinal of $V$ is the smallest positive $i$ verifying $$(t(\pi ))^i(x)=x.$$ 
A repeated use of Proposition \ref{biane} gives that, for such an $i$, 
\begin{equation} \label{cycl}
(t(\mbox{Red}_n^{(k)}(\pi )))^i(\mbox{Red}_n^{(k)}(x))=\mbox{Red}_n^{(k)}(x).
\end{equation}
Thus $i$ is a multiple of the cardinal of $\mbox{Red}_n^{(k)}(V)$, which is also characterized 
by the fact that it is the smallest positive $i$ veriying condition \eqref{cycl}.
\end{proof}

It is not so difficult to see that, 
if there is a block of multiplicity $k+1$ in $\pi\cup \mbox{Kr}(\pi )$, 
for $\pi \in NC^{(A)}((k+1)n)$, the other blocks are necessarily simple, 
because one cannot link two elements of the same congruence class 
without crossing the block of multiplicity $k+1$.
This is in fact a particular case of the following result : 

\begin{Cor} 
For $\pi \in NC^{(k)}(n)$,
$$\sum _{V\in \mbox{bl}(\pi \cup \mbox{Kr}(\pi ))} (mult_{\pi \cup \mbox{Kr}(\pi )}(V)-1)=k.$$
\end{Cor}

\begin{proof}
This is a simple computation.
First notice that 
\begin{eqnarray} \label{mult}
\nonumber \sum _{V\in \mbox{bl}(\pi \cup \mbox{Kr}(\pi ))} (mult_{\pi \cup \mbox{Kr}(\pi )}(V)-1) =
\\\sum _{V\in \mbox{bl}(\pi \cup \mbox{Kr}(\pi ))} mult_{\pi \cup \mbox{Kr}(\pi )}(V)-|\pi \cup \mbox{Kr}(\pi )|.
\end{eqnarray}
The first term in \eqref{mult} is 
$$\sum _{W\in \mbox{bl}(\mbox{Red}_n^{(k)}(\pi \cup \mbox{Kr}(\pi )))} 
\sum _{V\in \mbox{bl}(\pi \cup \mbox{Kr}(\pi )): \mbox{Red}_n^{(k)}(V)=W} mult_{\pi \cup \mbox{Kr}(\pi )}(V).$$
But for any block $W$ of $\mbox{Red}_n^{(k)}(\pi \cup \mbox{Kr}(\pi ))$, one has 
$$\sum _{V\in \mbox{bl}(\pi \cup \mbox{Kr}(\pi )): \mbox{Red}_n^{(k)}(V)=W} mult_{\pi \cup \mbox{Kr}(\pi )}(V)=k+1.$$
Applying twice formula \eqref{nbblocks}, we get 
\begin{align}\nonumber \sum _{V\in \mbox{bl}(\pi \cup \mbox{Kr}(\pi ))} (mult_{\pi \cup \mbox{Kr}(\pi )}(V)-1)=\ &(k+1)|\mbox{Red}_n^{(k)}(\pi \cup \mbox{Kr}(\pi ))|-|\pi \cup \mbox{Kr}(\pi )|
\\\nonumber =\ &(k+1)(n+1)-((k+1)n+1)
\\\nonumber =\ &k.
\end{align}
\end{proof}

For a partition $\pi \in NC^{(k)}(n)$, one may define 
a vector $\lambda _\pi $ with integer coordinates as follows : 
$$(\lambda _\pi )_i=\sum _{V\in \mbox{bl}(\pi \cup \mbox{Kr}(\pi )): \mbox{Red}_n^{(k)}(V)
=\mbox{Mix}(\mbox{Red}_n^{(k)}(\pi \cup \mbox{Kr}(\pi )),i)} (mult_{\pi \cup \mbox{Kr}(\pi )}(V)-1).$$
The vector $\lambda _\pi \in \Lambda _{n+1,k}$ is called the {\em shape} of $\pi $.

\begin{Rem} 
A type B non-crossing partition $\pi $ is determined by 
its absolute value $p:=\mbox{Abs}(\pi )$ and 
the choice of the block $Z\in \mbox{bl}(p\cup \mbox{Kr}(p))$, 
which has to be lifted to the zero-block of $\pi $. 
This latter choice is encoded in the shape $\lambda _\pi $ of $\pi $. 
Indeed, type B corresponds to the case $k=1$ of non-crossing 
partitions of type $k$ and therefore the shape $\lambda _\pi $ belongs 
to the set $\Lambda _{n+1,1}$ consisting of the $n+1$ vectors 
$e_i=(\delta _i^j)_{1\leq j\leq n+1} , 1\leq i\leq n+1$. 
That $\lambda _\pi =e_i$ means exactly that we have to choose the block 
$Mix(p,i)$ as the absolute value of the zero-block. 
The conclusion is that a type B non-crossing partition, 
considered as a non-crossing partition of type $1$, is 
determined by its reduction (or absolute value in the type B language) 
and its shape. Unfortunately, this is not the case when $k\geq 2$. 
It is interesting to ask how to determine a general 
non-crossing partition of type $k$. This question is investigated 
in the proof of the next proposition.
\end{Rem}

\begin{Prop} \label{const}
Let $\lambda \in \Lambda _{n+1,k}$. 
The number of $\pi \in NC^{(k)}(n)$ having shape $\lambda $ 
and reduction a fixed non-crossing partition $p\in NC^{(A)}(n)$ 
is the same for any choice of $p\in NC^{(A)}(n)$. 
We will denote this quantity by $r(\lambda )$.
\end{Prop}

\begin{proof}
As announced, we investigate how to determine a type $k$ non-crossing partition 
$\pi \in NC^{(k)}(n)$, once its reduction $p\in NC^{(A)}(n)$ 
and its shape $\lambda \in \Lambda _{n+1,k}$ are given. 
We know that $\mbox{Mix}(p,1)$ is a singleton of $[n]\cup \overline{[n]}$. 
For simplicity, we assume that it is a singleton $\{i\}$ of $[n]$. 
We need to know how to form the blocks of $\pi $ reducing to $\{i\}$. 
The number of admissible ways to form these blocks depends 
on the value of $\lambda _1$ but of course not on $p$, 
because the actual value of $i$ does not come into the game. 
Assume that these blocks are formed; this gives a decomposition 
of $[(k+1)n]\setminus \{x\mid \mbox{Red}_n^{(k)}(x)=i\}\cup \overline{[(k+1)n]}$ 
into $\lambda _1+1$ sets, according to the following process : 
let us denote by $\{i+l_1n,\ldots ,i+l_mn\}$ the smallest (with respect to $\sqsubset $) of the blocks 
we have just formed that is not simple (if there is no such block, i.e. when $\lambda _1=0$, 
our decomposition is trivial) ; each of the $\{\overline{i+l_jn},\ldots ,\overline{i-1+l_{j+1}n}\}$ 
becomes a set in our decomposition after erasing the 
$i+ln, l_j\leq l\leq l_{j+1}$, for each $1\leq j\leq m-1$. 
Then remove all elements $x$ such that $i+l_1n\leq x\leq i+l_mn$ and repeat the process 
by considering the new smallest block with respect to $\sqsubset $ 
among the remaining blocks that are not simple. 
Notice that the sets obtained this way may be identified 
with sets of the form $[l(n-1)]\cup \overline{[l(n-1)]}$, for some $l\leq k+1$, up to identifying the first and 
last elements of the sets. This can be done, because these elements are necessarily linked by 
construction of the Kreweras complement. 
On each of these sets, $\pi $ induces a non-crossing partition that belongs to $NC^{(l)}(n-1)$. 
All such induced non-crossing partitions have the same reduction $\tilde p$ obtained by erasing 
in $p\cup \mbox{Kr}(p)$ the element $i$ and by identifying $\overline{i-1}$ with $\overline{i}$ 
(which are also necessarily linked in $\mbox{Kr}(p)$). The shapes of the induced partitions sum 
to the shape $\lambda $ of $\pi $. 
Hence a non-crossing partition of type $k$ is determined by its reduction $p$, its shape $\lambda $, 
an admissible way to form the blocks reducing to $\mbox{Mix}(p,1)$, an admissible decomposition 
of $\lambda $ and the choice of the induced non-crossing partitions in sets $NC^{(l)}(n-1)$, 
having reduction $\tilde p$ and shape the summands in the decomposition of $\lambda $.\\
Our argument goes by induction on $n$. For $n=1$ and any $k$, there is only one possible reduction, 
because $\# NC^{(A)}(1)=1$ and consequently there is nothing to prove in that case. 
Assume that, for any $l$, the number of partitions in $NC^{(l)}(n-1)$ with given shape and reduction
does not depend on the choice of the reduction. 
According to our analysis of the first part of the proof, the number of partitions $\pi \in NC^{(k)}(n)$ with 
given shape $\lambda $ and reduction $p$ does not depend on the choice of the reduction, because we noticed that 
the number of admissible ways to form the blocks reducing to $\mbox{Mix}(p,1)$ does not depend on $p$, 
the shape decomposition depend only on $\lambda $ and the way the latter blocks are formed, 
and by induction, the numbers of choices for the induced partitions only depend on their shapes.
\end{proof}

\begin{Rem} 
For small values of $k$, one may easily compute the values of $r(\lambda )$ 
for each $\lambda \in \Lambda _{n+1,k}$.\\ 
In the simplest case $k=0$, $$\Lambda _{n+1,0}=\{(0,\ldots ,0)\},$$ $$r((0,\ldots ,0))=1.$$ 
For $k=1$, $$\Lambda _{n+1,1}=\{e_i\}_{i=1,\ldots ,n+1},$$ and one has 
$$\forall 1\leq i\leq n+1, r(e_i)=1.$$
For $k=2$, $$\Lambda _{n+1,2}=\{e_i+e_j\}_{i,j=1,\ldots ,n+1}.$$ 
The value of $r(e_i+e_j)$ depends on whether $i=j$ or not:
$$\forall 1\leq i\leq n+1, r(2e_i)=1.$$
$$\forall 1\leq i<j\leq n+1, r(e_i+e_j)=3.$$
\end{Rem}

\

We investigate in the next subsection some properties of the set $NC^{(k)}(n)$.

\begin{center}
{\bf 6.3 Study of the poset $NC^{(k)}(n)$}
\end{center}

The set $NC^{(k)}(n)$, being a subset of $(NC^{(A)}((k+1)n),\preceq )$, 
inherits its partially ordered set (abbreviated poset) structure. 
Contrary to $NC^{(B)}(n)$, which is a sublattice of $(NC^{(A)}(2n),\preceq )$ 
(up to the identification $[\pm n]=[2n]$), $(NC^{(k)}(n),\preceq )$ is unfortunately 
not a sublattice of $(NC^{(A)}((k+1)n),\preceq )$, when $k\geq 2$. 

\begin{Rem} \label{cex}
When $k=2$ and $n=2$, consider the partitions 
$$\pi :=\{\{2,3,4,5\},\{1,6\}\}\in NC^{(2)}(2)$$
and
$$\rho :=\{\{1,2\},\{3,4,5,6\}\}\in NC^{(2)}(2).$$
It is an easy exercise to determine the meet of these two partitions 
in the lattice $(NC^{(A)}(6),\preceq )$: 
$$\pi \wedge _{NC^{(A)}(6)} \rho =\{\{1\},\{2\},\{3,4,5\},\{6\}\}.$$
It is immediate that $\pi \wedge _{NC^{(A)}(6)} \rho $ is not an element 
of $NC^{(2)}(2)$ which is consequently not a sublattice of $(NC^{(A)}(6),\preceq )$ ; 
the same kind of argument would prove that $NC^{(k)}(n)$ is never a sublattice 
of $(NC^{(A)}((k+1)n),\preceq )$, as soon as $k\geq 2$. 
\end{Rem}

It is natural to ask whether $NC^{(k)}(n)$ is or not a lattice in its own right 
for the reverse refinement order $\preceq $. We do not know the answer to this question. 

We now state and prove the main result of this section.

\begin{Theo} \label{cover}
$\pi \mapsto \mbox{Red}_n^{(k)}(\pi )$ is a $\frac{1}{(k+1)n+1}C_{(n+1)(k+1)}^{k+1}$-to-$1$ 
map from $NC^{(k)}(n)$ onto $NC^{(A)}(n)$.
\end{Theo}

\begin{proof}
We fix $p\in NC^{(A)}(n)$. 
The shape $\lambda _\pi $ of a $\pi \in NC^{(k)}(n)$ 
satisfying $\mbox{Red}_n^{(k)}(\pi )=p$ is an element of the set $\Lambda _{n+1,k}$, 
and for each $\lambda \in \Lambda _{n+1,k}$, there are exactly $r(\lambda )$ 
non-crossing partitions of type $k$ with reduction $p$ and shape $\lambda $. 
Hence there are $\sum_{\lambda \in \Lambda _{n+1,k}} r(\lambda )$ 
non-crossing partitions of type $k$ with reduction $p$, and we know by 
Proposition \ref{const} that this number does not depend on $p$. 
It remains to prove that $\sum_{\lambda \in \Lambda _{n+1,k}} r(\lambda )=\frac{1}{(k+1)n+1}C_{(n+1)(k+1)}^{k+1}$, 
by counting the non-crossing partitions of type $k$ with reduction $1_{[n]}$. 
The set formed by these partitions is precisely the set $NC_n(k)$ of 
non-crossing partitions of $[(k+1)n]$ having blocks of size divisible by $n$. 
The latter set appears in \cite{bbcc}, where it is proved that its cardinal 
is $\frac{1}{(k+1)n+1}C_{(n+1)(k+1)}^{k+1}$. 
\end{proof}

We end this section by defining a subset of $NC^{(k)}(n)$ that will be used in Section 7.

\begin{Def} 
We write $NC_\ast ^{(k)}(n)$ for the set of non-crossing partitions of type $k$ 
without non-simple blocks in their Kreweras complement.
\end{Def} 

\begin{Rem} 
In the shape of a non-crossing partition $\pi \in NC_\ast ^{(k)}(n)$, the coordinates 
corresponding to blocks of $\mbox{Kr}(\pi )$ are zero ; there is therefore a straightforward 
bijection between the set of shapes of non-crossing partitions $\pi \in NC_\ast ^{(k)}(n)$ 
satisfying $\mbox{Red}_n^k(\pi )=p$ and the set $\Lambda _{|p|,k}$. 
Notice also that, given $p\in NC^{(A)}(n)$ 
and $\lambda \in \Lambda _{|p|,k}$, there are exactly $r(\lambda )$ non-crossing partitions 
$\pi \in NC_\ast ^{(k)}(n)$ with reduction $p$ and, with a small abuse of language, shape $\lambda $. 
\end{Rem}

Non-crossing partitions of type $k$ give a combinatorial description of the 
version of the boxed convolution with scalars in $\mathcal{C}_k$, 
as explained in the next section. 

\begin{center} 
{\bf\large 7. Boxed convolution of type $k$}
\end{center}

As for type A and B, there is a boxed convolution operation 
associated to the non-crossing partitions of type $k$. 
It is defined on formal power series with coefficients in $\CC^{k+1}$, as follows.

\begin{Def} 
\begin{enumerate}
\item We denote by $\Theta ^{(k)}$ the set of power series of the form 
$$f(z)=\sum_{n=1}^\infty  (\alpha _n^{(0)},\ldots ,\alpha _n^{(k)})z^n,$$ where, 
for each $n\geq 1$ and $0\leq i\leq k$, $\alpha _n^{(i)}$ is a complex number.
\item Let $f(z)=\sum_{n=1}^\infty  (\alpha _n^{(0)},\ldots ,\alpha _n^{(k)})z^n$ 
and $g(z)=\sum_{n=1}^\infty  (\beta _n^{(0)},\ldots ,\beta _n^{(k)})z^n$ be in $\Theta ^{(k)}$.
For every $m\geq 1$ and every $0\leq i\leq k$, consider the numbers $\gamma _m^{(i)}$ defined by
\begin{align}\nonumber \gamma _m^{(i)}=\sum_{\pi \in NC^{(i)}(m)} \ &
\frac{C_i^{(\lambda _\pi )_1,\ldots ,(\lambda _\pi )_{m+1}}}{r(\lambda _\pi )}
\prod_{j=1}^{ | \mbox{Red}_m^{(i)}(\pi ) | } 
\alpha _{\mbox{card}(\mbox{Sep}(\mbox{Red}_m^{(i)}(\pi ),j))}^{((\lambda _\pi )_j)}\cdot 
\\\nonumber \ &\prod_{j= | \mbox{Red}_m^{(i)}(\pi ) | +1}^{m+1}
\beta _{\mbox{card}(\mbox{Sep}(\mbox{Red}_m^{(i)}(\pi )),j)}^{((\lambda _\pi )_j)}.
\end{align}
Then the series $\sum_{n=1}^\infty  (\gamma _n^{(0)},\ldots ,\gamma _n^{(k)})z^n$ 
is called the boxed convolution of type $k$ of $f$ and $g$, 
and is denoted by $f$ \framebox[7pt]{$\star$}$^{(k)} g$.
\end{enumerate}
\end{Def}

It turns out that, up to identifying the two sets $\Theta ^{(k)}$ 
and $\Theta _{\mathcal{C}_k}^{(A)}$, the two operations \framebox[7pt]{$\star$}$^{(k)}$ 
and \framebox[7pt]{$\star$}$_{\mathcal{C}_k}^{(A)}$ are actually 
the same, as stated in the next theorem.

\begin{Theo} \label{starboxthm}
\framebox[7pt]{$\star$}$^{(k)}=$\framebox[7pt]{$\star$}$_{\mathcal{C}_k}^{(A)}$
\end{Theo}

\begin{proof}
Let $f(z)=\sum_{n=1}^\infty  (\alpha _n^{(0)},\ldots ,\alpha _n^{(k)})z^n$ 
and $g(z)=\sum_{n=1}^\infty  (\beta _n^{(0)},\ldots ,\beta _n^{(k)})z^n$ be in $\Theta ^{(k)}$.\\
Write $f$ \framebox[7pt]{$\star$}$^{(k)} g=\sum_{n=1}^\infty  (\gamma _n^{(0)},\ldots ,\gamma _n^{(k)})z^n$ and 
$f$ \framebox[7pt]{$\star$}$_{\mathcal{C}_k}^{(A)} g=\sum_{n=1}^\infty  (\delta _n^{(0)},\ldots ,\delta _n^{(k)})z^n$.
We fix a positive integer $n$, for which we will show that 
$$(\gamma _n^{(0)},\ldots ,\gamma _n^{(k)})=(\delta _n^{(0)},\ldots ,\delta _n^{(k)}).$$
Let us look at $\gamma _n^{(i)}$.
First, we have 
\begin{align}\nonumber \gamma _n^{(i)}=\sum_{\pi \in NC^{(i)}(n)} \ &
\frac{C_i^{(\lambda _\pi )_1,\ldots ,(\lambda _\pi )_{n+1}}}{r(\lambda _\pi )}
\prod_{j=1}^{ | \mbox{Red}_m^{(i)}(\pi ) | } 
\alpha _{\mbox{card}(\mbox{Sep}(\mbox{Red}_m^{(i)}(\pi ),j))}^{((\lambda _\pi )_j)}\cdot 
\\\nonumber \ &\prod_{j= | \mbox{Red}_m^{(i)}(\pi ) | +1}^{n+1}
\beta _{\mbox{card}(\mbox{Sep}(\mbox{Red}_m^{(i)}(\pi )),j)}^{((\lambda _\pi )_j)}.
\end{align}
For every $\pi \in NC^{(i)}(n)$, $1\leq j\leq n+1$ and $0\leq \lambda \leq k$, 
we put $p=\mbox{Red}_n^{(i)}(\pi )$ and
$$\theta ^{(\lambda )}(p,j):=\left\{\begin{array}{ccc} 
\alpha _{\mbox{card}(\mbox{Sep}(p,j))}^{(\lambda )}&if&j\leq | p | ,
\\\beta _{\mbox{card}(\mbox{Sep}(p,j))}^{(\lambda )}&if&j> | p | ,\\\end{array}\right.$$
The summation over $NC^{(i)}(n)$ can be reduced to one over $NC^{(A)}(n)$, 
by using the cover $\mbox{Red}_n^{(i)} : NC^{(i)}(n)\rightarrow NC^{(A)}(n)$.
When doing so, and taking into account the explicit description of 
$(\mbox{Red}_n^{(i)})^{-1}(p), p\in NC^{(A)}(n)$ provided by the proof of Theorem \ref{cover}, 
one gets $$\gamma _n^{(i)}=\sum_{p\in NC^{(A)}(n)} 
\sum_{\lambda \in \Lambda _{n+1,i}} C_i^{\lambda _1,\ldots ,\lambda _{n+1}} \prod_{j=1}^{n+1}\theta ^{(\lambda _j)}(p,j).$$
On the other hand, by recalling the definition of the operation 
\framebox[7pt]{$\star$}$_{\mathcal{C}_k}^{(A)}$, 
we see that $\delta _n^{(i)}$ equals $$\sum_{p\in NC^{(A)}(n)} \sum_{\lambda \in \Lambda _{n+1,i}} 
C_i^{\lambda _1,\ldots ,\lambda _{n+1}} \prod_{j=1}^{n+1}\theta ^{(\lambda _j)}(p,j).$$
By comparing, we obtain $(\gamma _n^{(0)},\ldots ,\gamma _n^{(k)})=(\delta _n^{(0)},\ldots ,\delta _n^{(k)})$, as desired.
\end{proof}

\begin{Cor} 
The operation \framebox[7pt]{$\star$}$^{(k)}$ is associative, 
commutative and the series $\Delta ^{(k)}(z)=\Delta _{\mathcal{C}_k}^{(A)}(z)$ is its unit element. 
A series $f\in \Theta ^{(k)}$ is invertible with respect to 
\framebox[7pt]{$\star$}$^{(k)}$ if and only if its coefficient of degree one has 
a non-zero first component.
\end{Cor}

\begin{Rem} 
Theorem \ref{starboxthm} tells us that the operation \framebox[7pt]{$\star$}$^{(k)}$ 
is a boxed convolution of type A, for which it is noticed in Remark \ref{multi} that 
one may define a generalization to power series in several noncommuting indeterminates. 
This means that there exists an operation \framebox[7pt]{$\star$}$^{(k)}$ 
on power series in several noncommuting indeterminates. 
We do not find interesting to record here the formulas involved in this operation. 
\end{Rem}

Non-crossing partitions of type $k$ are thus the combinatorial objects describing 
the version of the boxed convolution of type A with scalars in the algebra $\mathcal{C}_k$. 

It is now easy to rewrite the main formulas involving infinitesimal non-crossing 
cumulants with sums indexed by the set of non-crossing partitions of type $k$. 
This is the content of the next proposition : 

\begin{Prop} 
Let $(\mathcal{A}, (\varphi ^{(i)})_{0\leq i\leq k})$ 
be an infinitesimal noncommutative probability space of order $k$.
The infinitesimal non-crossing cumulant functionals satisfy, 
for every $n\geq 1$, every $0\leq i\leq k$ and every $a_1,\ldots ,a_n\in \mathcal{A}$ : 
$$\varphi ^{(i)}(a_1\cdots  a_n) = \sum_{\pi \in NC_\ast ^{(i)}(n)} 
\frac{C_i^{(\lambda _\pi )_1,\ldots ,(\lambda _\pi )_{|\pi |}}}{r(\lambda _\pi )}
\kappa _{\mbox{Red}_n^{(i)}(\pi )}^{(\lambda _\pi )}(a_1,\ldots ,a_n).$$
\end{Prop}

\begin{Prop} 
Let $(\mathcal{A}, (\varphi ^{(i)})_{0\leq i\leq k})$ be an 
infinitesimal noncommutative probability space of order $k$. 
Consider subsets $\mathcal{M}_1,\mathcal{M}_2$ of $\mathcal{A}$ 
that are infinitesimally free of order $k$. 
Then, one has, for each $n\geq 1$, 
each $n$-tuples $(a_1,\ldots ,a_n)\in \mathcal{M}_1^n,(b_1,\ldots ,b_n)\in \mathcal{M}_2^n$ 
and each $0\leq i\leq k$ :  
$$\kappa _n^{(i)}(a_1\cdot b_1,\ldots ,a_n\cdot b_n)=\sum_{\pi \in NC^{(i)}(n)}
\frac{C_i^{(\lambda _\pi )_1,\ldots ,(\lambda _\pi )_{n+1}}}{r(\lambda _\pi )}
\kappa _{\mbox{Red}_n^{(i)}(\pi \cup \mbox{Kr}(\pi ))}^{(\lambda _\pi )}(a_1,b_1,\ldots ,a_n,b_n).$$
\end{Prop}

We move to the main application of infinitesimal freeness. 

\begin{center} 
{\bf\large 8. Application to derivatives of the free convolution}
\end{center}

In this final section, we give an application of 
infinitesimal freeness of order $k$. We consider the situation already 
examined in \cite{bs} : let 
$\{a_u^v(t)\mid 1\leq v\leq m_u\}_{t\in K}$ 
be $s$ families of noncommutative random variables in a 
(usual) noncommutative probability space $(\mathcal{A},\varphi )$. 
These families are indexed by a subset $K$ of $\RR$ 
having zero as an accumulation point, and we are 
interested in the joint distribution $\mu _t$ 
of $\{a_u^v(t)\mid 1\leq v\leq m_u, 1\leq u\leq s\}$
when $t$ is going to $0$, in other words for infinitesimal values of $t$. 
Recall that $\mu _t$ is the linear functional on 
$\CC \langle X_u^v , 1\leq v\leq m_u , 1\leq u\leq s \rangle$ 
defined by : 
$$\mu _t(P((X_u^v)_{1\leq v\leq m_u , 1\leq u\leq s}))=
\varphi (P((a_u^v(t))_{1\leq v\leq m_u , 1\leq u\leq s})).$$
In what follows, we will consider a family $\{\mu _t\}_{t\in K}$ 
of linear functionals on $\CC \langle X_u^v , 1\leq v\leq m_u , 1\leq u\leq s \rangle$
without any further reference to the variables 
$\{a_u^v(t)\mid 1\leq v\leq m_u, 1\leq u\leq s\}_{t\in K}$.
For each value of $t\in K$, one may obviously define the non-crossing 
cumulant functionals $((\kappa _t)_n: 
(\CC \langle X_u^v , 1\leq v\leq m_u , 1\leq u\leq s \rangle )^n \rightarrow \CC)_{n=1}^\infty $ 
associated to the noncommutative 
probability space $(\CC \langle X_u^v , 1\leq v\leq m_u , 1\leq u\leq s \rangle , \mu _t)$. 
A way to capture the behavior of $\mu _t$ for infinitesimal values of $t$ 
is to introduce recursively its derivatives at $0$ by : 
\begin{equation} \label{zeroderiv}
\mu ^{(0)}:=\lim_{t\rightarrow 0} \mu _t,
\end{equation}
\begin{equation} \label{ithderiv}
\frac{\mu ^{(i)}}{i!}:=\lim_{t\rightarrow 0} \frac{1}{t^i}(\mu _t-\sum_{j=0}^{i-1} \frac{t^j}{j!}\mu ^{(j)}),1\leq i\leq k.
\end{equation}
We will assume that the limits in formulas \eqref{zeroderiv} and \eqref{ithderiv} exist 
and use the notation $\mu ^{(i)}=\frac{d^i}{dt^i}_{|t=0}\mu _t$. 
Notice that, in \cite{bs}, only $\mu ^{(0)}$ and $\frac{d}{dt}_{|t=0}\mu _t$ were studied. 
It follows from formulas \eqref{zeroderiv} and \eqref{ithderiv} that 
$$\mu _t=\sum_{i=0}^k \frac{\mu ^{(i)}}{i!} t^i+o(t^k).$$
Notice that $(\mu ^{(i)})_{0\leq i\leq k}$ is an infinitesimal law 
(of order $k$) on $\sum_{u=1}^s m_u$ variables and therefore 
$(\CC \langle X_u^v , 1\leq v\leq m_u , 1\leq u\leq s \rangle , (\mu ^{(i)})_{0\leq i\leq k})$ 
is an infinitesimal noncommutative probability space of order $k$. 
Associated to this infinitesimal noncommutative probability space of order $k$, 
we have infinitesimal non-crossing cumulant functionals 
$(\kappa _n^{(i)}: \mathcal{A}^n \rightarrow \CC , 0\leq i\leq k)_{n=1}^\infty $, 
as defined by formula \eqref{inffmc}. 
These infinitesimal cumulant functionals are linked to $((\kappa _t)_n)_{n=1}^\infty $ 
as follows : 

\begin{Prop} \label{propcumtime}
For every $n\geq 1$ and $0\leq i\leq k$, 
$$\kappa _n^{(i)}=\frac{d^i}{dt^i}_{|t=0}(\kappa _t)_n.$$
\end{Prop}

\begin{proof}
By the inverse of the free moment-cumulant formula, one has 
\begin{equation} \label{timecum}
\forall t\in K, (\kappa _t)_n=\sum_{p\in NC^{(A)}(n)} \mbox{M\"ob}(p,1_n) (\mu _t)_p.
\end{equation}
By the assumption made above, the right-hand side of 
formula \eqref{timecum} has $k$ derivatives at $0$, 
hence $\frac{d^i}{dt^i}_{|t=0}(\kappa _t)_n$ is well-defined and, 
using linearity of derivation and Leibniz rule, 
one obtains : 
$$\frac{d^i}{dt^i}_{|t=0}(\kappa _t)_n=\sum_{\substack{p\in NC^{(A)}(n)\\p:=\{V_1,\ldots ,V_h \} }}
\sum_{\lambda \in \Lambda _{h,i}} \mbox{M\"ob}(p,1_n) C_i^{\lambda _1,\ldots ,\lambda _h} \mu _{p}^{(\lambda )}.$$
One recognizes in the right-hand side above the right-hand side of formula \eqref{infinvfmc}, 
and we are done.
\end{proof}

This proposition will be the main tool to characterize infinitesimal 
freeness of order $k$ in terms of moments in Theorem \ref{inffmomentsthm}. 
We first give a recipe to deduce the infinitesimal behaviour 
of the free convolution of two families of distributions from their 
individual infinitesimal behaviours. 

\begin{Prop} 
Let $\{\mu _t\}_{t\in K}$ (resp. $\{\nu _t\}_{t\in K}$) 
be a family of linear functionals on $\CC \langle X_u , 1\leq u\leq m \rangle$
(resp $\CC \langle Y_u , 1\leq u\leq m \rangle$) 
such that $\mu ^{(i)}=\frac{d^i}{dt^i}_{|t=0}\mu _t$ 
(resp. $\nu ^{(i)}=\frac{d^i}{dt^i}_{|t=0}\nu _t$) exist for $0\leq i\leq k$. 
Set : $$(\eta ^{(i)})_{0\leq i\leq k}:=(\mu ^{(i)})_{0\leq i\leq k} \boxplus ^{(k)}(\nu ^{(i)})_{0\leq i\leq k},$$
$$(\theta ^{(i)})_{0\leq i\leq k}:=(\mu ^{(i)})_{0\leq i\leq k} \boxtimes ^{(k)}(\nu ^{(i)})_{0\leq i\leq k}.$$
Then $\eta ^{(i)}=\frac {d^i}{dt^i}\mid _{t=0} \mu _t \boxplus \nu _t$ 
and $\theta ^{(i)}=\frac {d^i}{dt^i}\mid _{t=0} \mu _t \boxtimes \nu _t$.
\end{Prop}

\begin{proof}
For each $t\in K$, we consider the free product 
$$(\CC \langle X_u, Y_u , 1\leq u\leq m \rangle,\mu _t\star \nu _t).$$
Since $\frac{d^i}{dt^i}_{|t=0}\mu _t$ and $\frac{d^i}{dt^i}_{|t=0}\nu _t$ 
exist by assumption for each $0\leq i\leq k$, 
we obtain the existence of $\frac{d^i}{dt^i}_{|t=0}(\mu _t\star \nu _t)$ 
for each $0\leq i\leq k$ and these functionals 
are completely determined by the $\mu ^{(i)}$'s and the $\nu ^{(i)}$'s. 
In the infinitesimal noncommutative probability space 
$(\CC \langle X_u, Y_u , 1\leq u\leq m \rangle ,
(\frac{d^i}{dt^i}_{|t=0}(\mu _t\star \nu _t))_{0\leq i\leq k})$, 
the unital subalgebras $\mathcal{A}_1=\CC \langle X_u , 1\leq u\leq m \rangle$ 
and $\mathcal{A}_2=\CC \langle Y_u , 1\leq u\leq m \rangle$ 
are infinitesimally free of order $k$ : 
indeed, if $n\geq 1 , 0\leq i\leq k$ and $P_1\in \mathcal{A}_{i_1},\ldots ,P_n\in \mathcal{A}_{i_n}$ 
are such that $i_1,\ldots ,i_n$ are not all equal, 
then $$\kappa _n^{(i)}(P_1,\ldots ,P_l)=\frac{d^i}{dt^i}_{|t=0}(\kappa _t)_n(P_1,\ldots ,P_l),$$
where $(\kappa _t)_n$ is the $n$-th non-crossing cumulant functional 
in the noncommutative probability space 
$(\CC \langle X_u, Y_u , 1\leq u\leq m \rangle,\mu _t\star \nu _t)$, 
by Proposition \ref{propcumtime}. 
But it follows from the construction of the free product 
that $(\kappa _t)_n(P_1,\ldots ,P_l)=0$ for each $t\in K$. 
In particular $\kappa _n^{(i)}(P_1,\ldots ,P_l)=0$. 
The infinitesimal distribution of the $m$-tuple 
$(X_1+Y_1,\ldots ,X_m+Y_m)$ (resp. $(X_1\cdot Y_1,\ldots ,X_m\cdot Y_m)$)
is, on the one hand $(\frac{d^i}{dt^i}_{|t=0}(\mu _t\boxplus \nu _t))_{0\leq i\leq k}$
(resp. $(\frac{d^i}{dt^i}_{|t=0}(\mu _t\boxtimes \nu _t))_{0\leq i\leq k}$) 
by construction of the free product and, on the other hand, 
$(\eta ^{(i)})_{0\leq i\leq k}$ (resp. $(\theta ^{(i)})_{0\leq i\leq k}$) 
by the argument above. 
\end{proof}

We conclude by a characterization of infinitesimal freeness of order $k$ 
in terms of moments. Its formulation and proof rely on the 
Proposition \ref{propcumtime}. 

\begin{Theo} \label{inffmomentsthm}
Let $(\mathcal{A}, (\varphi ^{(i)})_{0\leq i\leq k})$ be an 
infinitesimal noncommutative probability space of order $k$, and 
$\mathcal{A}_1,\ldots ,\mathcal{A}_n$ be unital subalgebras of $\mathcal{A}$. 
Then $\mathcal{A}_1,\ldots ,\mathcal{A}_n$ are infinitesimally free 
of order $k$ if and only if for any positive integer $l\in \NN^*$, 
and any $a_1\in \mathcal{A}_{i_1},\ldots , 
a_l\in \mathcal{A}_{i_l}$, one has 
\begin{equation} \label{inffmomentstime}
\varphi _t((a_1-\varphi _t(a_1))\cdots  (a_l-\varphi _t(a_l)))=o(t^k),
\end{equation}
whenever $i_1\not=\ldots \not=i_l$, where $\varphi _t:=\sum_{i=0}^k \frac{\varphi ^{(i)}}{i!} t^i$. 
The condition \eqref{inffmomentstime} translates into $k+1$ requirements : 
\begin{equation} \label{inffmoments}
\forall i\in \{0,\ldots ,k\}, \sum_{j=0}^i \sum_{\lambda \in \Lambda _{l,i-j}} (-1)^{\#\{m, \lambda _m>0\}} 
\mu ^{(j)}(\hat \mu ^{(\lambda _1)}(P_1)\cdots  \hat \mu ^{(\lambda _{l})}(P_l))=0,
\end{equation}
where $\hat \mu ^{(\lambda )}(P):=P-\mu ^{(0)}(P)$ if $\lambda =0$, 
and $\hat \mu ^{(\lambda )}(P):=\mu ^{(\lambda )}(P)$ else.\\
\end{Theo}

\begin{proof}
We assume that condition \eqref{inffmomentstime} holds and have to prove that 
$\mathcal{A}_1,\ldots ,\mathcal{A}_n$ 
satisfy the vanishing of mixed infinitesimal cumulants condition. 
Using Proposition \ref{propcumtime}, it is equivalent to prove that 
for $l\geq 2$, and $a_1\in \mathcal{A}_{i_1},\ldots , 
a_l\in \mathcal{A}_{i_l}$
\begin{equation} \label{vmcumtime}
(\kappa _t)_l(a_1,\ldots ,a_l)=o(t^k)
\end{equation}
whenever $\exists r\not=s, i_r\not=i_s$, where $(\kappa _t)_l$ is the 
$l$-th non-crossing cumulant functional in $(\mathcal{A}, \varphi _t)$.\\
We proceed by induction on $l\geq 2$.\\
It is easy to see that 
\begin{equation} \label{fmc2}
(\kappa _t)_2(a_1,a_2)=\varphi _t((a_1-\varphi _t(a_1))((a_2-\varphi _t(a_2)).
\end{equation}
If $a_1\in \mathcal{A}_{i_1}, 
a_2\in \mathcal{A}_{i_2}$ 
with $i_1\not=i_2$, the right-hand side of \eqref{fmc2} is $o(t^k)$ by assumption. 
We assume then that the vanishing of mixed infinitesimal cumulants 
is proved for $2,3,\ldots ,l-1$ variables, and consider 
$(\kappa _t)_l(a_1,\ldots ,a_l)$ with 
$a_1\in \mathcal{A}_{i_1},\ldots , 
a_l\in \mathcal{A}_{i_l}$
such that $\exists r\not=s, i_r\not=i_s$. 
By Propositions \ref{cwsae}, \ref{cwpae} and the induction hypothesis, 
we may assume that $\varphi _t(a_1)=\ldots =\varphi _t(a_l)=0$ and $i_1\not=\ldots \not=i_l$. 
Write then the free moment-cumulant formula : 
$$\forall t\in K, (\varphi _t)(a_1\cdots a_l)-\sum_{\substack{p\in NC^{(A)}(l)\\p\not=1_l}} 
(\kappa _t)_p(a_1,\ldots ,a_l) = (\kappa _t)_l(a_1,\ldots ,a_l).$$
By assumption, $(\varphi _t)(a_1\cdots a_l)=o(t^k)$. 
Any non-crossing partition $p\not =1_l$ owns 
an interval-block $V_0$, as noticed in Section 3. 
If $V_0$ is a singleton, 
$$(\kappa _t)_p(a_1,\ldots ,a_l) = 
(\varphi _t)_{|V_0|}((a_1,\ldots ,a_l) \mid V_0) \prod_{V\not =V_0} (\kappa _t)_{|V|}((a_1,\ldots ,a_l) \mid V)=0.$$
Otherwise, $V_0$ contains two following, hence distinct, indices, 
and, by induction hypothesis, 
$$(\kappa _t)_{|V_0|}((a_1,\ldots ,a_l) \mid V_0)=o(t^k).$$
Since, for each $V\in \mbox{bl}(p)$, $(\kappa _t)_{|V|}((a_1,\ldots ,a_l) \mid V)$ 
is bounded in a neighborhood of $0$, one may affirm that 
$$(\kappa _t)_p(a_1,\ldots ,a_l) = o(t^k).$$
We conclude that 
$$(\kappa _t)_l(a_1,\ldots ,a_l)=o(t^k),$$
as required.\\
For the converse, we assume that the 
vanishing of mixed infinitesimal cumulants is satisfied, 
or equivalently that equation \eqref{vmcumtime} holds. 
We write then the free moment-cumulant formula : 
\begin{eqnarray} \label{fmcumtimebis} 
\forall t\in K, (\varphi _t)(a_1-\varphi _t(a_1)\cdots a_l-\varphi _t(a_l)) = 
\\\label{fmcumtimeter}
\sum_{p\in NC^{(A)}(l)} (\kappa _t)_p(a_1-\varphi _t(a_1),\ldots ,a_l-\varphi _t(a_l)).
\end{eqnarray}
If $a_1\in \mathcal{A}_{i_1},\ldots , 
a_l\in \mathcal{A}_{i_l}$ 
with $i_1\not=\ldots \not=i_n$, the same argument as above gives that 
\eqref{fmcumtimeter} is $o(t^k)$. 
This concludes the proof.
\end{proof}

{\bf Acknowledgements.}
The author would like to express all his gratitude to Serban Teodor Belinschi for suggesting this problem, 
and for many useful discussions. This work was initiated during a stay at University of Saskatchewan, 
that the author thanks for the warm welcome and the excellent work conditions provided. 
This work was partially supported by the {\emph Agence Nationale de la
Recherche} grant ANR-08-BLAN-0311-03.

\end{document}